\documentclass{amsart}

\usepackage{ae,aecompl}
\usepackage{chngcntr}
\usepackage{graphicx}
\usepackage[all,cmtip]{xy}
\usepackage{amsmath, amscd}
\usepackage{amsthm}
\usepackage{amssymb}
\usepackage{amsfonts}
\usepackage{qsymbols}
\usepackage{latexsym}
\usepackage{cite}
\usepackage{color}
\usepackage{url}

\newcommand {\abs}[1]{\lvert#1\rvert}
\newcommand {\A}{{\mathcal{A}}}

\newcommand{\Be}{{B}}

\newcommand {\C}{{\mathbb C}}

\newcommand {\Ce}{\mathrm{C}}

\newcommand {\D}{D}

\newcommand {\ud}{\mathrm{d}}
\newcommand {\ue}{e}
\newcommand {\E}{{\mathcal{E}}}
\newcommand {\Ell}{L}
\newcommand{\Ellp}{{L^{p}}} 
\newcommand{\Ellq}{{L^{q}}}
\newcommand{\Ellr}{L^{r}}

\newcommand{\Ellinfty}{L^{\infty}}
\newcommand{\Wnp}{W^{n,p}} 

\newcommand{\Wonep}{W^{1,p}}
\newcommand{\Wonepeq}{W^{1,p}} 
\newcommand {\F}{{\mathcal{F}}}
\newcommand {\HT}{{\mathcal{H}}}
\newcommand {\ui}{i}

\newcommand{\ind}{{\mathbf{1}}}

\newcommand {\La}{{\mathcal{L}}}

\newcommand {\M}{{\mathcal{M}}}

\newcommand {\N}{{{\mathbb N}}}
\newcommand {\norm}[1]{\left\|#1\right\|}
\newcommand {\ph}{{\varphi}}
\newcommand {\R}{{\mathbb R}}
\newcommand {\Rad}{{\mathrm{Rad}}}
\newcommand {\ran}{{\mathrm{ran}}}
 
\newcommand {\supp}{{\mathrm{supp}}}
\newcommand{\Se}{\mathrm{S}}
\newcommand{\St}{{\mathrm{St}}} 
\newcommand{\Sw}{\mathcal{S}}

\newcommand {\w}{{\omega}}
\newcommand {\W}{{\mathrm{W}}}
\newcommand {\Z}{{{\mathbb Z}}}
\newcommand {\vanish}[1]{\relax}

\newcommand {\DAPQ}{{\D_{A}(\frac{1}{p}-\frac{1}{q},1)}}
\newcommand {\DAPQBIG}{{\D_{A}(\tfrac{1}{p}-\tfrac{1}{q},1)}}
\newcommand {\DLOGPQ}{{\D_{\log(A)}(\frac{1}{p}-\frac{1}{q},1)}}
\newcommand {\DLOGPQBIG}{{\D_{\log(A)}(\tfrac{1}{p}-\tfrac{1}{q},1)}}

\DeclareMathOperator{\Real}{Re}
\DeclareMathOperator{\Imag}{Im}

\newtheorem{theorem}{Theorem}[section]
\newtheorem{lemma}[theorem]{Lemma}
\newtheorem{proposition}[theorem]{Proposition}
\newtheorem{corollary}[theorem]{Corollary}

\theoremstyle{definition}
\newtheorem{definition}[theorem]{Definition}
\newtheorem{remark}[theorem]{Remark}
\newtheorem{example}[theorem]{Example}

\numberwithin{equation}{section}

\title{Functional calculus for $C_{0}$-groups using type and cotype}

\author{Jan Rozendaal}
\address{Mathematical Sciences Institute\\ Australian National University\\ Acton ACT 2601\\ Australia\\ and Institute of Mathematics, Polish Academy of Sciences\\
ul.~\'{S}niadeckich 8\\
00-656 Warsaw\\
Poland}
\email{janrozendaalmath@gmail.com}


\date{}

\dedicatory{}


\begin{document}

\begin{abstract}
We study the functional calculus properties of generators of $C_{0}$-groups under type and cotype assumptions on the underlying Banach space. In particular, we show the following. Let $-\ui A$ generate a $C_{0}$-group on a Banach space $X$ with type $p\in[1,2]$ and cotype $q\in[2,\infty)$. Then $f(A):(X,\D(A))_{\frac{1}{p}-\frac{1}{q},1}\to X$ is bounded for each bounded holomorphic function $f$ on a sufficiently large strip. 
As a corollary of this result, for sectorial operators we quantify the gap between bounded imaginary powers and a bounded $\HT^{\infty}$-calculus in terms of the type and cotype of the underlying Banach space. For cosine functions we obtain similar results as for $C_{0}$-groups. We extend our theorems to $R$-bounded operator-valued calculi, and we give an application to the theory of rational approximation of $C_{0}$-groups.
\end{abstract}

\maketitle

\section{Introduction}\label{introduction}

Let $-\ui A$ generate a $C_{0}$-group $(U(s))_{s\in\R}\subseteq\La(X)$ on a Banach space $X$, and set 
\begin{align}\label{group type}
\theta(U):=\inf\left\{\theta\geq 0\left|\exists M\geq 1:\norm{U(s)}\leq M\ue^{\theta\abs{s}}\textrm{ for all }s\in\R\right.\right\}.
\end{align}
For each $\w>0$ let
\begin{align}\label{strip}
\St_{\w}:=\left\{z\in\C\left|\,\abs{\Imag(z)}>\w\right.\right\},
\end{align}
and let $\HT^{\infty}\!(\St_{\w})$ be the space of bounded holomorphic functions on $\St_{\w}$ with the supremum norm. There is a natural definition of $f(A)$ as an unbounded operator on $X$ for each $\w>\theta(U)$ and each $f\in\HT^{\infty}\!(\St_{\w})$ (see Section \ref{functional calculus}). It was shown by Boyadzhiev and deLaubenfels in \cite{Boyadzhiev-deLaubenfels94} that, if $X$ is a Hilbert space, then there exists a constant $C\geq 0$ such that $f(A)\in\La(X)$ with 
\begin{align}\label{Hinfty estimate}
\norm{f(A)}_{\La(X)}\leq C\norm{f}_{\HT^{\infty}\!(\St_{\w})}
\end{align}
for each $\w>\theta(U)$ and $f\in\HT^{\infty}\!(\St_{\w})$. One says that $A$ has a \emph{bounded $\HT^{\infty}$-calculus}. It is useful to know that an operator has a bounded $\HT^{\infty}$-calculus. The theory of $\HT^{\infty}$-calculus has applications to questions of maximal regularity (see \cite{Arendt04,Kalton-Weis01,Kunstmann-Weis04}) and played a crucial role in the solution to the Kato square root problem in \cite{AuHoLaMcITc02,AxKeMcI06}. Also, functional calculus bounds can be used to determine convergence rates in the numerical approximation theory for solutions to evolution equations (see \cite{Brenner-Thomee79,Gomilko-Tomilov14, Egert-Rozendaal13}). 

One can obtain nontrivial functional calculus results for $C_{0}$-groups even when $X$ is not a Hilbert space. For example, it was shown in \cite{Haase09} that if $-\ui A$ generates a $C_{0}$-group $(U(s))_{s\in\R}$ on a UMD space $X$, then $A$ has a bounded $\HT^{\infty}_{1}\!(\St_{\w})$-calculus for all $\w>\theta(U)$. Here $\HT^{\infty}_{1}\!(\St_{\w})$ consists of all $f\in\HT^{\infty}\!(\St_{\w})$ such that
\begin{align}\label{mikhlin norm}
\norm{f}_{\HT^{\infty}_{1}\!(\St_{\w})}:=\sup_{z\in\St_{\w}}\,\abs{f(z)}+(1+\abs{z})\abs{f'(z)}<\infty.
\end{align}
In \cite{Haase-Rozendaal16} a similar statement was obtained on general Banach spaces, where one restricts the calculus to real interpolation spaces between $X$ and the domain of $A$.

However, there are several drawbacks to functional calculus theory for function spaces which are strictly contained in the class of $\HT^{\infty}$-functions. For example, in applications one might be interested in functions $f$ which are not contained in a smaller function space such as $\HT^{\infty}_{1}\!(\St_{\w})$. But even when dealing with a function $f$ which is known to yield a bounded operator $f(A)$, it can be of interest to know that \eqref{Hinfty estimate} holds, instead of an estimate for $\norm{f(A)}_{\La(X)}$ with respect to a larger function norm. This is the case for numerical approximation schemes for solutions to evolution equations, where one can often improve convergence rates for operators with a bounded $\HT^{\infty}$-calculus (see \cite{Egert-Rozendaal13,FlNeWe03}).

Moreover, a bounded $\HT^{\infty}$-calculus yields square function estimates that arise frequently in harmonic analysis (see \cite{Kalton-Weis04, Haak-Haase13, LeMerdy04}) and that are of use in the theory of stochastic evolution equations \cite{vanNeerven10, vNeVeWe12}. Also, a bounded $\HT^{\infty}$-calculus can be bootstrapped to yield large classes of $R$-bounded operators, and the notion of $R$-boundedness has various applications in the theory of evolution equations (see \cite{Kunstmann-Weis04}). It seems that such connections are not available in full generality for subspaces of the class of $\HT^{\infty}$-functions (if $A$ is self-adjoint then more can be said; see for example \cite{AuMcIMo15}).

Finally, it seems somewhat unsatisfactory that there is such a rough division between functional calculus properties on Hilbert spaces and on UMD spaces, in the sense that one goes from a bounded $\HT^{\infty}$-calculus on Hilbert spaces to a bounded $\HT^{\infty}_{1}$-calculus on UMD spaces. One might expect a finer division of functional calculus theorems within the class of UMD spaces, and on $\Ellp$-spaces one might hope that functional calculus properties improve as $p$ tends to $2$. The latter is indeed the case for symmetric contraction semigroups on $\Ellp$-spaces, cf.~\cite{Carbonaro-Dragicevic17, Cowling83}.

In the present article we aim to address the issues above. We study $\HT^{\infty}$-calculus for generators of $C_{0}$-groups in terms of the type and cotype of the underlying Banach space (see Definition \ref{type and cotype}). Our main result, proved as Theorem \ref{main theorem}, is as follows. (For the real interpolation space $\DAPQ$ see \eqref{real interpolation}).

\begin{theorem}\label{main result introduction}
Let $-\ui A$ generate a $C_{0}$-group $(U(s))_{s\in\R}\subseteq\La(X)$ on a Banach space $X$ with type $p\in[1,2]$ and cotype $q\in[2,\infty)$. Let $\w>\theta(U)$. Then there exists a constant $C\geq 0$ such that $\DAPQBIG\subseteq\D(f(A))$ and
\begin{align*}
\norm{f(A)x}_{X}\leq C\norm{f}_{\HT^{\infty}\!(\St_{\w})}\norm{x}_{\DAPQ}
\end{align*}
for all $f\in\HT^{\infty}\!(\St_{\w})$ and $x\in \DAPQ$.
\end{theorem}

Under additional geometric assumptions we improve Theorem \ref{main result introduction}. Indeed, in Theorem \ref{second main theorem} we show that, if $X$ is isomorphic to a complemented subspace of a $p$-convex and $q$-concave Banach lattice for $p\in[1,2]$ and $q\in[2,\infty)$, then for each $\lambda>\w>\theta(U)$ there exists a constant $C\geq 0$ such that $\D((\lambda+\ui A)^{-\frac{1}{p}+\frac{1}{q}})\subseteq\D(f(A))$ and
\begin{align}\label{second main estimate introduction}
\norm{f(A)x}_{X}\leq C\norm{f}_{\HT^{\infty}\!(\St_{\w})}\|(\lambda+\ui A)^{-\frac{1}{p}+\frac{1}{q}}x\|_{X}
\end{align}
for all $f\in\HT^{\infty}\!(\St_{\w})$ and $x\in \D((\lambda+\ui A)^{-\frac{1}{p}+\frac{1}{q}})$.

One might say that Theorem \ref{main result introduction} shows that each generator $-\ui A$ of a $C_{0}$-group on a Banach space $X$ with type $p$ and cotype $q$ has a bounded $\HT^{\infty}$-calculus from $\DAPQ$ to $X$, and $A$ has a bounded $\HT^{\infty}$-calculus from $\D((\lambda+\ui A)^{-\frac{1}{p}+\frac{1}{q}})$ to $X$ under additional assumptions on $X$. Since $\frac{1}{p}-\tfrac{1}{q}=0$ if and only if $X$ is isomorphic to an $L^{2}$-space, \eqref{second main estimate introduction} recovers as a special case \eqref{Hinfty estimate}, the main result of \cite{Boyadzhiev-deLaubenfels94}. 

The interpolation spaces $\D_{A}(\theta,1)$ and the fractional domains $\D((\lambda+\ui A)^{-\theta})$ increase as $\theta$ tends to zero. Therefore the statements in this paper show that many functional calculus properties depend in a quantitative manner on how close the geometry of the underlying space is to that of a Hilbert space. Since each UMD space has non-trivial type and finite cotype, Theorem \ref{main result introduction} provides a scale of functional calculus results on UMD spaces. 

For $(\Omega,\mu)$ a measure space and $p\in[1,\infty)$, the space $\Ellp(\Omega,\mu)$ is a $p$-convex and $p$-concave Banach lattice with type $\min(p,2)$ and cotype $\max(p,2)$. Hence our results show that the functional calculus properties of group generators on $\Ellp$-spaces improve as $p$ tends to $2$.  It should be noted that \eqref{second main estimate introduction} holds without any assumptions on the Banach space if $\tfrac{1}{p}-\tfrac{1}{q}\geq \tfrac{1}{2}$ (see Remark \ref{rem:trivial results}). Therefore most of the results in Sections \ref{results for groups} and \ref{sectorial and cosine} are only of interest on $L^{p}$-spaces when $p\in(1,\infty)$. 

In Proposition \ref{functions with decay} we deduce from Theorem \ref{main result introduction} that each $f\in\HT^{\infty}\!(\St_{\w})$ with polynomial decay of order $\alpha>\frac{1}{p}-\frac{1}{q}$ at infinity satisfies $f(A)\in\La(X)$. Under the assumptions of \eqref{second main estimate introduction} the case $\alpha=\frac{1}{p}-\frac{1}{q}$ is attained. 

It was shown by McIntosh in \cite{McIntosh86} that, on Hilbert spaces, sectorial operators with bounded imaginary powers have a bounded sectorial $\HT^{\infty}$-calculus. It is also known that on general Banach spaces (even on $L^{p}$-spaces) an operator with bounded imaginary powers need not have a bounded $\HT^{\infty}$-calculus. In Theorem \ref{BIP theorem} we quantify the gap between bounded imaginary powers and a bounded $\HT^{\infty}$-calculus, by showing that each sectorial operator $A$ with bounded imaginary powers has a bounded sectorial $\HT^{\infty}$-calculus from $\DLOGPQBIG$ to $X$ if the underlying Banach space $X$ has type $p\in[1,2]$ and cotype $q\in[2,\infty)$. As in the classical case of a bounded $\HT^{\infty}$-calculus on $X$, one obtains from this an unconditionality result and square function estimates.

For generators of cosine functions we derive similar results as for $C_{0}$-groups. In particular, in Theorem \ref{cosine function result} we show that each generator $-A$ of a cosine function on a Banach space $X$ with type $p$ and cotype $q$ has a bounded parabola-type $\HT^{\infty}$-calculus from $\D_{A}(\tfrac{1}{2}(\tfrac{1}{p}-\tfrac{1}{q}),1)$ to $X$. 

We extend Theorem \ref{main result introduction} and \eqref{second main estimate introduction} to operator-valued functional calculi. Then in Theorem \ref{R-bounded calculus} we show that $A$ in fact has an $R$-bounded $\HT^{\infty}$-calculus from $\D_{A}(\frac{1}{p}-\frac{1}{q},1)$ to $X$ if $X$ additionally has property $(\alpha)$. In Theorem \ref{second R-bounded calculus} we obtain an $R$-bounded version of \eqref{second main estimate introduction}. These results are sharp, as Example \ref{sharpness example} shows.

To indicate the use of Theorem \ref{main result introduction} we give an application to the study of numerical approximation methods for the solutions to evolution equations. In Proposition \ref{convergence rates} we consider a method of rational approximation proposed in \cite{JaNeOz12,NeOoSa13}, and for exponentially stable groups on $\Ell^{p}$-spaces we improve the rates of convergence which were obtained in \cite{Egert-Rozendaal13}. In Corollary \ref{other approximation methods} we show that many rational approximation methods of exponentially stable $C_{0}$-groups converge strongly on $\D_{A}(\tfrac{1}{p}-\tfrac{1}{q},1)$ if the underlying space $X$ has type $p\in[1,2]$ and cotype $q\in[2,\infty)$.

To prove Theorem \ref{main result introduction} we use transference principles going back to \cite{Calderon68,Coifman-Weiss77,BeGiMu89}. The transference technique has been applied to functional calculus theory in \cite{Cowling83}, \cite{Hieber-Pruss98} and \cite{Haase09,Haase11,Haase-Rozendaal13,Haase-Rozendaal16}. In \cite{Haase-Rozendaal16} an interpolation version of the transference principle for unbounded groups from \cite{Haase09} was established. This transference principle was then combined with a theorem about Fourier multipliers on vector-valued Besov spaces. In the present paper we also obtain a transference principle involving vector-valued Besov spaces, but we combine it with a theorem about Fourier multipliers between distinct Besov spaces. For \eqref{second main estimate introduction} we use statements about Fourier multipliers from $\Ellp$ to $\Ellq$ which were obtained recently in \cite{Rozendaal-Veraar17,Rozendaal-Veraar18}. We suspect that it is possible to apply these Fourier multiplier theorems to the transference principle from \cite{Haase-Rozendaal13} for $C_{0}$-semigroups, but this matter will not be explored in the present article.

This paper is organized as follows. In Section \ref{functional calculus and Fourier multipliers} we discuss some of the basics of functional calculus theory, vector-valued Besov spaces and the notions of type and cotype. We then state two Fourier multiplier results which are vital for that which follows. In Section \ref{transference principles} we establish two new transference principles for $C_{0}$-groups. These are used in Section \ref{results for groups} to prove Theorem \ref{main result introduction} and \eqref{second main estimate introduction}, as well as several corollaries for $C_{0}$-groups. In Section \ref{sectorial and cosine} we obtain results for sectorial operators and generators of cosine functions, and in Section \ref{operator-valued calculus} we extend the results from previous sections to $R$-bounded operator-valued calculi. Finally, in Section \ref{rational approximation} we give an application  to the theory of rational approximation schemes.

\subsection{Notation and terminology}\label{notation and terminology}

We write $\N:=\left\{1,2,3,\ldots\right\}$ for the natural numbers and $\N_{0}:=\N\cup\left\{0\right\}$. We write $\C_{+}$ for the open right half-plane of complex numbers $z\in\C$ with $\Real(z)>0$, and $\C_{-}:=\C\setminus\overline{\C_{+}}$.

We denote nonzero Banach spaces over the complex numbers by $X$ and $Y$. The space of bounded linear operators from $X$ to $Y$ is $\La(X,Y)$, and $\La(X):=\La(X,X)$. The domain of a closed operator $A$ on $X$ is $\D(A)$, a Banach space with the norm
\begin{align*}
\norm{x}_{\D(A)}:=\norm{x}_{X}+\norm{Ax}_{X}\qquad(x\in \D(A)).
\end{align*}
The spectrum of $A$ is $\sigma(A)$ and the resolvent set $\rho(A):=\C\setminus \sigma(A)$. We write $R(\lambda,A):=(\lambda-A)^{-1}$ for the resolvent operator of $A$ at $\lambda\in\rho(A)$.

For $p\in[1,\infty]$ and $(\Omega,\mu)$ a measure space, $\Ellp(\Omega;X)$ denotes the Bochner space of equivalence classes of $p$-integrable $X$-valued functions on $\Omega$. We often write $\norm{\cdot}_{p}=\norm{\cdot}_{\Ellp(\R;\C)}$. For $n\in\N_{0}$ we let $\Wnp(\R;X)$ be the Sobolev space of $n$ times weakly differentiable $f\in\Ellp(\R;X)$ such that $f^{(n)}\in\Ellp(\R;X)$. The H\"{o}lder conjugate of $p$ is denoted by $p'$ and is defined by $1=\frac{1}{p}+\frac{1}{p'}$.

We let $\M(\R)$ denote the space of complex Borel measures on $\R$ with the total variation norm. For $\w\geq 0$ we let $\M_{\w}(\R)$ be the convolution algebra of $\mu\in\M(\R)$ of the form $\mu(\ud s)=\ue^{-\w \abs{s}}\nu(\ud s)$ for some $\nu\in\M(\R)$, with $\|\mu\|_{\M_{\w}\!(\R)}:=\|\ue^{\w\abs{\cdot}}\mu\|_{\M(\R)}$.

For $\Omega\neq \emptyset$ open in $\C$, we denote by $\HT^{\infty}\!(\Omega)$ the space of bounded holomorphic functions $f:\Omega\to\C$, a Banach algebra with the norm
\begin{align*}
\norm{f}_{\HT^{\infty}\!(\Omega)}:=\sup_{z\in\Omega}\,\abs{f(z)}\qquad(f\in\HT^{\infty}\!(\Omega)).
\end{align*}

The class of $X$-valued rapidly decreasing smooth functions on $\R$ is $\Sw(\R;X)$, and the space of $X$-valued tempered distributions is $\Sw'(\R;X)$. The Fourier transform of
$\Phi\in\Sw'(\R;X)$ is $\F\Phi$. If
$\mu\in\M_{\w}(\R)$ for $\w>0$ then $\F\mu\in \HT^{\infty}\!(\St_{\w})$ is given by
\begin{align*} 
\F\mu(z):=\int_{\R}\ue^{-\ui z s}\mu(\ud{s})\qquad(z\in\St_{\w}).
\end{align*} 

An interpolation couple is a pair $(X,Y)$ of Banach spaces which are embedded continuously in a Hausdorff topological vector space $Z$.
The real interpolation space of $(X,Y)$ with parameters $\theta\in [0,1]$ and $q\in [1,\infty]$ is denoted by $(X,Y)_{\theta,q}$. If $T:X+Y\rightarrow X+Y$ restricts to a bounded operator on $X$ and $Y$ then 
\begin{align}\label{interpolation inequality}
\norm{T}_{\La((X,Y)_{\theta,q})}\leq \norm{T}_{\La(X)}^{1-\theta}
\norm{T}_{\La(Y)}^{\theta}
\end{align}
for all $\theta\in (0,1)$ and $q\in[1,\infty]$. We mainly consider real interpolation spaces for the interpolation couple $(X,\D(A))$, where $A$ is a closed operator on $X$. We write 
\begin{align}\label{real interpolation}
\D_{A}(\theta,q):=(X,\D(A))_{\theta,q}\quad\textrm{and}\quad\norm{x}_{\theta,q}:=\norm{x}_{\D_{A}(\theta,q)}\quad(x\in\D_{A}(\theta,q)).
\end{align} 

For an operator $B$ on $X$ and a continuously embedded space $Y\hookrightarrow X$, the part of $B$ in $Y$ is the operator $B_{Y}$ on $Y$ that satisfies $B_{Y}y=By$ for $y\in \D(B_{Y}):=\left\{z\in \D(B)\cap Y\mid Bz\in Y\right\}$. We write $B_{\theta,q}:=B_{\D_{A}(\theta,q)}$ for $\theta\in[0,1]$ and $q\in[1,\infty]$.

\section{Functional calculus and Fourier multipliers}\label{functional calculus and Fourier multipliers}

In this section we present the background on functional calculus and Fourier multipliers which will be needed for the rest of the article.

\subsection{Functional calculus}\label{functional calculus}

We assume that the reader is familiar with the basics of the theory of $C_{0}$-groups from \cite{Engel-Nagel00}. For more on the functional calculus for generators of $C_{0}$-groups see \cite[Chapter 4]{Haase06a}.

An operator $A$ on a Banach space $X$ is a \emph{strip-type operator} of \emph{height} $\w_{0}\geq 0$ if $\sigma(A)\subseteq \overline{\St_{\w_{0}}}$, where $\St_{0}:=\R$, and $\sup_{\lambda\in\C\setminus \St_{\w}}\norm{R(\lambda,A)}<\infty$ for all $\w>\w_{0}$. For $\w>0$ set
\begin{align*}
\E(\St_{\w}):=\left\{g\in \HT^{\infty}\!(\St_{\w})\left| \text{$g(z)\in
O(\abs{z}^{-\alpha})$ for some $\alpha>1$ as $\abs{\Real(z)}\rightarrow \infty$}\right.\right\}.
\end{align*}
The \emph{strip-type functional calculus} for a strip-type operator $A$ of height $\w_{0}$ is defined as follows. First, operators $f(A)\in \La(X)$ are associated with $f\in \E(\St_{\w})$ for $\w>\w_{0}$:
\begin{align}\label{Cauchy integral}
f(A):=\frac{1}{2\pi \ui}\int_{\partial \St_{\w'}}f(z)R(z,A)\,\ud z.
\end{align}
Here $\partial \St_{\w'}$ is the positively oriented boundary of $\St_{\w'}$ for $\w'\in (\w_{0},\w)$. This procedure is independent of the choice of $\w'$ by Cauchy's theorem, and yields an algebra homomorphism $\E(\St_{\w})\rightarrow \La(X)$, $f\mapsto f(A)$. The definition of $f(A)$ is extended to a larger class of functions by \emph{regularization}:
\begin{align}\label{regularization}
f(A):=e(A)^{-1}(ef)(A)
\end{align}
if there exists an $e\in \E(\St_{\w})$ with $e(A)$ injective and $ef\in \E(\St_{\w})$. Then $f(A)$ is a closed unbounded operator on $X$, and the definition of $f(A)$ is independent of the choice of the regularizer $e$.  Each $f\in \HT^{\infty}\!(\St_{\w})$ is regularizable by the function $z\mapsto(\lambda-z)^{-2}$ for $\abs{\Imag(\lambda)}>\w$.

Let $-\ui A$ generate a $C_{0}$-group $(U(s))_{s\in\R}\subseteq\La(X)$. Then $A$ is a strip-type operator of height $\theta(U)$, with $\theta(U)$ as in \eqref{group type}. Let $M\geq 1$ and $\w\geq 0$ be such that $\norm{U(s)}\leq M\ue^{\w\abs{s}}$ for all $s\in\R$, and for $\mu\in\M_{\w}(\R)$ set
\begin{align}\label{Hille-Phillips calculus}
U_{\mu}x:=\int_{\R}U(s)x\,\ud\mu(s)\qquad(x\in X).
\end{align}
The mapping $\mu\mapsto U_{\mu}$ is an algebra homomorphism $\M_{\w}(\R)\to\La(X)$ called the \emph{Hille-Phillips calculus}, and the following lemma from \cite[Lemma 2.2]{Haase07} shows that this calculus is consistent with the strip-type calculus for $A$.

\begin{lemma}\label{decay implies fourier transform}
Let $-\ui A$ generate a $C_{0}$-group $(U(s))_{s\in\R}\subseteq\La(X)$ on a Banach space $X$, and let $\w> \alpha>\theta(U)$. Then each $f\in \E(\St_{\w})$ satisfies $f=\F\mu$ and $f(A)=U_{\mu}\in\La(X)$ for some $\mu\in \M_{\alpha}(\R)$. Conversely, if $\mu\in\M_{\alpha}(\R)$ then $f:=\F\mu\in\HT^{\infty}(\St_{\alpha})$ is such that $f(A)=U_{\mu}\in\La(X)$.
\end{lemma}

In fact, it is observed in Remark \ref{rem:trivial results} that the first statement in Lemma \ref{decay implies fourier transform} holds for all $f\in\HT^{\infty}\!(\St_{\w})$ such that $f(z)=O(\abs{z}^{-1/2})$ as $\abs{\Real(z)}\to\infty$.

A fundamental result in functional calculus theory is the Convergence Lemma. The following is a version of this lemma adapted to our setting.

\begin{lemma}[Convergence Lemma]\label{convergence lemma}
Let $A$ be a densely defined strip-type operator of height $\w_{0}\geq 0$ on a Banach space $X$. Let $Y$ be a Banach space continuously embedded in $X$ such that $\D(A^{2})\subseteq Y$ is dense. Let $\w>\w_{0}$ and let $(f_{j})_{j\in J}\subseteq \HT^{\infty}\!(\St_{\w})$ be a net satisfying the following conditions:
\begin{itemize}
\item $\sup_{j\in J}\norm{f_{j}}_{\HT^{\infty}\!(\St_{\w})}<\infty$;
\item $f(z):=\lim_{j}f_{j}(z)$ exists for all $z\in \St_{\w}$;
\item $\sup_{j\in J}\norm{f_{j}(A)}_{\La(Y,X)}<\infty$.
\end{itemize}
Then $f\in\HT^{\infty}\!(\St_{\w})$, $f(A)\in\La(Y,X)$, $f_{j}(A)x\rightarrow f(A)x$ for all $x\in Y$ and 
\begin{align*}
\norm{f(A)}_{\La(Y,X)}\leq \limsup_{j\in J}\norm{f_{j}(A)}_{\La(Y,X)}.
\end{align*}
\end{lemma}
\begin{proof}
The proof is similar to the proofs of \cite[Proposition 5.1.7]{Haase06a} and \cite[Theorem 3.1]{BaHaMu13}. Vitali's theorem for nets from \cite[Theorem 2.1]{Arendt-Nikolski2000} implies that $f\in\HT^{\infty}\!(\St_{\w})$ and that $f_{j}\to f$ uniformly on compact subsets of $\St_{\w}$. Let $\lambda>\w$. Applying the dominated convergence theorem to \eqref{Cauchy integral} yields
\begin{align*}
f_{j}(A)x=\left(\frac{f_{j}(\cdot)}{(\ui\lambda-\cdot)^{2}}\right)(A)(\ui\lambda-A)^{2}x\to \left(\frac{f(\cdot)}{(\ui\lambda-\cdot)^{2}}\right)(A)(\ui\lambda-A)^{2}x=f(A)x
\end{align*}
and 
\begin{align*}
\norm{f(A)x}_{X}\leq \limsup_{j}\norm{f_{j}(A)}_{\La(Y,X)}\norm{x}_{Y}
\end{align*}
for all $x\in \D(A^{2})$. The required statements now follow since $\D(A^{2})\subseteq Y$ is dense and $f(A)$ is a closed operator on $X$.
\end{proof}

\subsection{Fourier multipliers and Banach space geometry}\label{Fourier multipliers and cotype}

In this section we treat Fourier multiplier operators under geometric assumptions on the underlying space. For more on the prerequisite notions from Banach space geometry, as well as for proofs of some of the statements below, see e.g.~\cite{DiJaTo95,HyNeVeWe16,HyNeVeWe17,Lindenstrauss-Tzafriri79}. Recall that a standard complex Gaussian random variable is a random variable $\gamma$ on a probability space $(\Omega,\mathbb{P})$ such that $\gamma=\frac{\gamma_{r}+\ui\gamma_{i}}{\sqrt{2}}$ for independent standard real Gaussian random variables $\gamma_{r},\gamma_{i}$ on $\Omega$. A \emph{Gaussian sequence} is a sequence of independent standard complex Gaussian random variables. 

\begin{definition}\label{type and cotype}
Let $X$ be a Banach space, $(\gamma_{k})_{k\in\N}$ a Gaussian sequence on a probability space $(\Omega,\mathbb{P})$ and $p\in[1,2]$, $q\in[2,\infty]$.
\begin{itemize}
\item $X$ has \emph{(Gaussian) type} $p$ if there exists a constant $C\geq 0$ such that for all $n\in\N$ and $x_{1},\ldots, x_{n}\in X$,
\begin{align}\label{type}
\Big(\mathbb{E}\Big\|\sum_{k=1}^{n}\gamma_{k}x_{k}\Big\|^{2}\Big)^{1/2}\leq C\Big(\sum_{k=1}^{n}\|x_{k}\|^{p}\Big)^{1/p}.
\end{align}
\item $X$ has \emph{(Gaussian) cotype} $q$ if there exists a constant $C\geq 0$ such that for all $n\in\N$ and $x_{1},\ldots, x_{n}\in X$,
\begin{align}\label{cotype}
\Big(\sum_{k=1}^{n}\|x_{k}\|^{q}\Big)^{1/q}\leq C\Big(\mathbb{E}\Big\|\sum_{k=1}^{n}\gamma_{k}x_{k}\Big\|^{2}\Big)^{1/2},
\end{align}
with the obvious modification for $q=\infty$.
\end{itemize}
\end{definition}

The minimal constants $C$ in \eqref{type} and \eqref{cotype} are called the \emph{Gaussian type $p$ constant} and the \emph{Gaussian cotype $q$ constant} and will be denoted by $\tau_{p,X}$ and $c_{q,X}$. We say that $X$ has \emph{nontrivial type} if $X$ has type $p\in(1,2]$ and that $X$ has \emph{finite cotype} if $X$ has cotype $q\in[2,\infty)$. By the Kahane-Khintchine inequalities, one may replace the exponent $2$ in \eqref{type} and \eqref{cotype} by any $r\in[1,\infty)$. This does not change the properties of type and cotype, only the minimal constants in \eqref{type} and \eqref{cotype}.

It is common to replace the Gaussian sequence in Definition \ref{type and cotype} by a \emph{Rademacher sequence}, i.e.\ a sequence $(r_{k})_{k\in\N}$ of independent complex random variables on a probability space $(\Omega,\mathbb{P})$ that are uniformly distributed on $\{z\in\C\mid \abs{z}=1\}$. This does not change the class of spaces under consideration, only the minimal constants in \eqref{type} and \eqref{cotype}. We choose to work with Gaussian sequences because the constants $\tau_{p,X}$ and $c_{q,X}$ occur in Proposition \ref{Fourier multipliers under type and cotype}. 

Every Banach space $X$ has type $p=1$ and cotype $q=\infty$, with $\tau_{1,X}=c_{\infty,X}=1$. If $X$ has type $p$ and cotype $q$ then it has type $r$ with $\tau_{r,X}\leq \tau_{p,X}$ for all $r\in[1,p]$ and cotype $s$ with $c_{s,X}\leq c_{q,X}$ for all $s\in[q,\infty]$. A Banach space $X$ has type $p=2$ and cotype $q=2$ if and only if $X$ is isomorphic to a Hilbert space, by Kwapie\'{n}'s result \cite{Kwapien72}. A Banach space $X$ with nontrivial type has finite cotype. Each UMD space has nontrivial type. Let $X$ be a Banach space with type $p\in[1,2]$ and cotype $q\in[2,\infty)$, let $r\in[1,\infty)$ and let $\Omega$ be a measure space. Then $\Ellr(\Omega;X)$ has type $\min(p,r)$ and cotype $\max(q,r)$ with $\tau_{\min(p,r),\Ellr(\Omega;X)}\leq C_{p,r}\tau_{p,X}$ and $c_{\max(q,r),\Ellr(\Omega;X)}\leq C_{q,r}c_{q,X}$ for constants $C_{p,r},C_{q,r}\geq0$ coming from the Kahane-Khintchine inequalities. 

Let $\psi\in \Ce^{\infty}\!(\R)$ be such that $\psi\geq0$, $\supp(\psi)\subseteq[\frac{1}{2},2]$ and $\sum_{k=-\infty}^{\infty}\psi(2^{-k}s)=1$ for all $s\in(0,\infty)$. For $k\in\N$ and $s\in\R$ let $\ph_{k}(s):=\psi(2^{-k}\abs{s})$ and $\ph_{0}(s):=1-\sum_{k=1}^{\infty}\ph_{k}(s)$. Let $X$ be a Banach space and let $p,q\in[1,\infty]$ and $r\in\R$. The \emph{(inhomogeneous) Besov space} $\Be^{r}_{p,q}(\R;X)$ is the space of all $f\in\Sw'(\R;X)$ such that
\begin{align*}
\norm{f}_{\Be^{r}_{p\!,q}(\R;X)}:=\Big\|\Big(2^{kr}\big\|\F^{-1}\ph_{k}\ast f\big\|_{\Ellp(\R;X)}\Big)_{k\in\N_{0}}\Big\|_{\ell^{q}}<\infty,
\end{align*}
endowed with the norm $\norm{\cdot}_{\Be^{r}_{p\!,q}(\R;X)}$. Then $\Be^{r}_{p,q}(\R;X)$ is a Banach space, $\Sw(\R;X)\subseteq \Be^{r}_{p,q}(\R;X)$ is dense if $p,q<\infty$, and a different choice of $\psi$ yields an equivalent norm on $\Be^{r}_{p,q}(\R;X)$. More details on vector-valued Besov spaces can be found in \cite{Amann97,Schmeisser-Sickel05}.

For $X$ a Banach space and $m\in\Ellinfty(\R;\La(X))$, the \emph{Fourier multiplier operator} $T_{m}:\Sw(\R;X)\to\Sw'(\R;X)$ with \emph{symbol} $m$ is given by
\begin{align*}
T_{m}(f):=\F^{-1}\left(m\cdot\F f\right)\qquad(f\in\Sw(\R;X)).
\end{align*}
For each $\mu\in\M(\R)$, 
\begin{align}\label{convolution}
L_{\mu}(f):=\mu\ast f\qquad(f\in\Sw(\R;X))
\end{align}
defines a Fourier multiplier operator with symbol $\F\mu\in\Ellinfty(\R)$. We say that $X$ is a \emph{UMD space} if the function $\ind_{[0,\infty)}-\ind_{(-\infty,0)}$ is the symbol of a bounded Fourier multiplier on $\Ell^{2}(\R;X)$.

Let $X$ and $Y$ be Banach spaces and let $(r_{k})_{k\in\N}$ be a Rademacher sequence on a probability space $(\Omega,\mathbb{P})$. A collection $\mathcal{T}\subseteq\La(X,Y)$ is \emph{$R$-bounded} if there exists a constant $C\geq 0$ such that, for all $n\in\N$, $x_{1},\ldots, x_{n}\in X$ and $T_{1},\ldots, T_{n}\in\mathcal{T}$, 
\begin{align}\label{R-bound}
\Big(\mathbb{E}\Big\|\sum_{k=1}^{n}r_{k}T_{k}x_{k}\Big\|_{Y}^{2}\Big)^{1/2}\leq C\Big(\mathbb{E}\Big\|\sum_{k=1}^{n}r_{k}x_{k}\Big\|_{X}^{2}\Big)^{1/2}.
\end{align}
The minimal constant $C$ in \eqref{R-bound} is denoted by $R(\mathcal{T})$. If we want to specify the underlying spaces $X$ and $Y$ then we write $R_{X,Y}(\mathcal{T})=R(\mathcal{T})$, and we write $R_{X}(\mathcal{T})=R(\mathcal{T})$ if $\mathcal{T}\subseteq\La(X)$. By the Kahane contraction principle, each uniformly bounded collection $\mathcal{T}\subseteq \C$ is $R$-bounded as a subset of $\La(X)$, with $R_{X}(\mathcal{T})$ equal to the uniform bound of $\mathcal{T}$.

The following result was obtained in \cite[Theorem 1.1]{Rozendaal-Veraar17}. 

\begin{proposition}\label{Fourier multipliers under type and cotype}
Let $X$ be a Banach space with type $p\in[1,2]$ and cotype $q\in[2,\infty]$. Let $r\in\R$, $s\in[1,\infty]$ and $m\in\Ellinfty(\R;\La(X))$ be such that $\{m(s)\mid s\in\R\}\subseteq\La(X)$ is $R$-bounded. Then $T_{m}\in\La(\Be^{r+1/p-1/q}_{p,s}(\R;X),\Be^{r}_{q,s}(\R;X))$ and
\begin{align*}
\norm{T_{m}}_{\La(\Be^{r+1/p-1/q}_{p,s}(\R;X),\Be^{r}_{q,s}(\R;X))}\leq 4^{\frac{1}{p}-\frac{1}{q}}\tau_{p,X}c_{q,X}\,R_{X}(\{m(s)\mid s\in\R\}).
\end{align*}
\end{proposition}

\begin{corollary}\label{Fourier multiplier measure}
Let $X$ be a Banach space with type $p\in[1,2]$ and cotype $q\in[2,\infty]$, and let $\mu\in\M(\R)$. Then $L_{\mu}\in\La(\Be^{1/p-1/q}_{p,1}(\R;X),\Ellq(\R;X))$ and
\[
\norm{L_{\mu}}_{\La(\Be^{1/p-1/q}_{p,1}(\R;X),\Ellq(\R;X))}\leq 4^{\frac{1}{p}-\frac{1}{q}}\tau_{p,X}c_{q,X}\norm{\F\mu}_{\Ellinfty(\R)}.
\]
\end{corollary}
\begin{proof}
By the Kahane contraction principle, $R_{X}(\{\F\mu(s)\mid s\in\R\})=\norm{\F\mu}_{\Ell^{\infty}(\R)}$. Since $\Be^{0}_{q,1}(\R;X)$ is contractively embedded in $\Ell^{q}(\R;X)$, the result follows by applying Proposition \ref{Fourier multipliers under type and cotype} to $L_{\mu}=T_{\F\mu}$.
\end{proof}

\begin{remark}\label{rem:better multiplier with UMD}
If in Corollary \ref{Fourier multiplier measure} one assumes additionally that $X$ is a UMD space, then $L_{\mu}\in \La(\Be^{1/p-1/q}_{p,p}(\R;X),\Ellq(\R;X))$ and
\begin{equation}\label{eq:UMD}
\norm{L_{\mu}}_{\La(\Be^{1/p-1/q}_{p,p}(\R;X),\Ellq(\R;X))}\leq C\tau_{p,X}c_{q,X}\norm{\F\mu}_{\Ellinfty(\R)}
\end{equation}
for each $\mu\in\M(\R)$ and some $C\geq 0$ independent of $\mu$. This follows from Proposition \ref{Fourier multipliers under type and cotype} and the embedding $\Be^{0}_{q,p}(\R;X)\subseteq L^{q}(\R;X)$ from \cite[Proposition 3.1]{Veraar13}.
\end{remark}

We assume that the reader is familiar with the basics of Banach lattices from \cite{Lindenstrauss-Tzafriri79}.

\begin{definition}\label{convexity and concavity}
Let $X$ be a Banach lattice and $p,q\in[1,\infty]$.
\begin{itemize}
\item $X$ is \emph{$p$-convex} if there exists a constant $C\geq 0$ such that for all $n\in\N$ and $x_{1},\ldots, x_{n}\in X$,
\begin{align*}
\Big\|\Big(\sum_{k=1}^{n}\abs{x_{k}}^{p}\Big)^{1/p}\Big\|\leq C\Big(\sum_{k=1}^{n}\|x_{k}\|^{p}\Big)^{1/p},
\end{align*}
with the obvious modification for $p=\infty$.
\item $X$ is \emph{$q$-concave} if there exists a constant $C\geq 0$ such that for all $n\in\N$ and $x_{1},\ldots, x_{n}\in X$,
\begin{align*}
\Big(\sum_{k=1}^{n}\|x_{k}\|^{q}\Big)^{1/q}\leq C\Big\|\Big(\sum_{k=1}^{n}\abs{x_{k}}^{q}\Big)^{1/q}\Big\|,
\end{align*}
with the obvious modification for $q=\infty$.
\end{itemize}
\end{definition}

Every Banach lattice $X$ is $1$-convex and $\infty$-concave. If $X$ is $p$-convex and $q$-concave then it is $r$-convex and $s$-concave for all $r\in[1,p]$ and $s\in[q,\infty]$. By \cite[Proposition 1.f.3]{Lindenstrauss-Tzafriri79}, if $X$ is $q$-concave then it has cotype $\max(q,2)$, and if $X$ is $p$-convex and $q$-concave for some $q<\infty$ then $X$ has type $\min(p,2)$. For $(\Omega,\mu)$ a measure space and $r\in[1,\infty]$, $\Ellr(\Omega,\mu)$ is an $r$-convex and $r$-concave Banach lattice.

A subspace $X_{0}\subseteq Y$ of a Banach space $Y$ is said to be \emph{complemented} if there exists a projection $P\in\La(Y)$ with $P(Y)=X_{0}$.

\begin{proposition}\label{Fourier multiplier Bessel space}
Let $X$ be isomorphic to a complemented subspace of a $p$-convex and $q$-concave Banach lattice, for $p\in[1,2]$ and $q\in[2,\infty)$. Then there exists a constant $C\geq 0$ such that $T_{m}\in\La(\Ellp(\R;X),\Ellq(\R;X))$ with
\begin{align}\label{Fourier multiplier inequality Bessel}
\norm{T_{m}}_{\La(\Ellp(\R;X),\Ellq(\R;X))}\leq CR_{X}(\{\abs{s}^{\frac{1}{p}-\frac{1}{q}}m(s)\mid s\in\R\})
\end{align}
for each $m\in\Ellinfty(\R;\La(X))$ such that $R_{X}(\{\abs{s}^{\frac{1}{p}-\frac{1}{q}}m(s)\mid s\in\R\})<\infty$.
\end{proposition}
\begin{proof}
Let $X_{0}$ be a subspace of a $p$-convex and $q$-concave Banach lattice $Y$, $S:X\to X_{0}$ an isomorphism and $P\in\La(Y)$ a projection with $P(Y)=X_{0}$. Let $m\in\Ell^{\infty}(\R;\La(X))$ be such that $R_{X}(\{\abs{s}^{\frac{1}{p}-\frac{1}{q}}m(s)\mid s\in\R\})<\infty$, and let $m_{0}\in\Ell^{\infty}(\R;\La(Y))$ be given by $m_{0}(s):=Sm(s)S^{-1}P\in\La(Y)$ for $s\in\R$. Then
\begin{align*}
R_{Y}(\{\abs{s}^{\frac{1}{p}-\frac{1}{q}}m_{0}(s)\mid s\in\R\})\leq \|S\|\|S^{-1}\|\|P\|R_{X}(\{\abs{s}^{\frac{1}{p}-\frac{1}{q}}m(s)\mid s\in\R\}).
\end{align*}
It follows from \cite[Theorem 3.21]{Rozendaal-Veraar18} that there exists a constant $C\geq 0$ independent of $m$ such that $T_{m_{0}}\in\La(\Ell^{p}(\R;Y),\Ell^{q}(\R;Y))$ with
\begin{align*}
\norm{T_{m_{0}}}_{\La(\Ell^{p}(\R;Y),\Ell^{q}(\R;Y))}\leq C R_{Y}(\{\abs{s}^{\frac{1}{p}-\frac{1}{q}}m_{0}(s)\mid s\in\R\}).
\end{align*}
This concludes the proof since $S^{-1}T_{m_{0}}S=T_{m}$.
\end{proof}

\section{Transference principles}\label{transference principles}

In this section we establish two new transference principles for $C_{0}$-groups, both based on the transference principles for unbounded groups from \cite{Haase09} and \cite{Haase-Rozendaal16}. 

For $\w\geq 0$ and $\mu\in\M_{\w}(\R)$ let $\mu_{\w}\in\M(\R)$ be given by
\begin{equation}\label{perturbed measure}
\mu_{\w}(\ud s):=\cosh(\w s)\mu(\ud s),
\end{equation}
and note that
\begin{equation}\label{Fourier transform perturbed measure}
\F\mu_{\w}(s)=\frac{\F\mu(s+\ui \w)+\F\mu(s-\ui \w)}{2}
\end{equation}
for all $s\in\R$.

\begin{proposition}\label{transference principle groups interpolation result}
Let $\w>\w_{0}\geq 0$, $p\in[1,2]$ and $q\in[2,\infty)$. Then there exists a constant $C\geq 0$ such that the following holds. Let $X$ be a Banach space with type $p$ and cotype $q$, and let $-\ui A$ generate a $C_{0}$-group $(U(s))_{s\in\R}\subseteq\La(X)$ such that $\norm{U(s)}_{\La(X)}\leq M\cosh(\w_{0}s)$ for all $s\in\R$ and some $M\geq 1$. Then
\begin{align*}
\norm{\int_{\R}U(s)x\,\mu(\ud s)}_{X}\leq C\tau_{p,X}c_{q,X}M^{2}\|\F\mu_{\w}\|_{\Ellinfty\!(\R)}\|x\|_{1/p-1/q,1}
\end{align*}
for all $\mu\in\M_{\w}(\R)$ and all $x\in\DAPQBIG$.
\end{proposition}
\begin{proof}
Since $\D_{A}(0,1)=\left\{0\right\}$ we may assume that $\frac{1}{p}-\frac{1}{q}\in(0,1)$ (for $\tfrac{1}{p}-\tfrac{1}{q}=0$ a stronger result is obtained in Proposition \ref{transference principle groups fractional domain result} below). Let
\begin{align}\label{define psi and phi}
\psi(s):=\frac{1}{\cosh(2\w s)}\qquad\textrm{and}\qquad\ph(s):=\frac{\sqrt{8}\w}{\pi}\frac{\cosh(\w s)}{\cosh(2\w s)}
\end{align}
for $s\in\R$. Define $\iota:X\to\Ellp(\R;X)$ by
\begin{align}\label{define iota}
\iota x(s):=\psi(-s)U(-s)x\qquad(x\in X,s\in\R)
\end{align}
and $P:\Ellq(\R;X)\to X$ by
\begin{align}\label{define P}
Pf:=\int_{\R}\ph(s)U(s)f(s)\,\ud s\qquad(f\in\Ellq(\R;X)).
\end{align}
Then $\iota$ is bounded and
\begin{align}\label{norm iota on Lp}
\norm{\iota}_{\La(X,\Ellp(\R;X))}&\leq M\norm{\psi(\cdot)\cosh(\w_{0}\cdot)}_{p}.
\end{align}
By H\"{o}lder's inequality, $P$ is bounded and
\begin{align}\label{norm P}
\norm{P}_{\La(\Ellq(\R;X),X)}&\leq M\norm{\ph(\cdot)\cosh(\w_{0}\cdot)}_{q'}.
\end{align}
Let $x\in\D(A)$. Then $\iota x\in\Ce^{1}(\R;X)$ and
\begin{align*}
(\iota x)'(s)&=-\psi'(-s)U(-s)x+\ui\psi(-s)U(-s)Ax\\
&=-2\w\frac{\tanh(2\w s)}{\cosh(2\w s)}U(-s)x+\ui \frac{1}{\cosh(2\w s)}U(-s)Ax
\end{align*}
for all $s\in\R$. Hence $(\iota x)'\in\Ellp(\R;X)$ and
\begin{align*}
\norm{(\iota x)'}_{p}\leq 2\w M\norm{\tanh}_{\Ellinfty(\R)}\norm{\frac{\cosh(\w_{0}\cdot)}{\cosh(2\w\cdot)}}_{p}\!\norm{x}_{X}+M\norm{\frac{\cosh(\w_{0}\cdot)}{\cosh(2\w \cdot)}}_{p}\!\norm{Ax}_{X}.
\end{align*}
Now \eqref{norm iota on Lp} implies that $\iota x\in\Wonep(\R;X)$, with
\begin{align*}
\norm{\iota x}_{1,p}\leq M(2\w\norm{\tanh}_{\Ellinfty(\R)}+1)\norm{\frac{\cosh(\w_{0}\cdot)}{\cosh(2\w\cdot)}}_{p}\!\norm{x}_{\D(A)}.
\end{align*}
Hence $\iota:\D(A)\rightarrow \Wonep(\R;X)$ is bounded and
\begin{align}\label{norm iota on Wnp}
\norm{\iota}_{\La(\D(A),\Wonepeq(\R;X))}\leq M(2\w\norm{\tanh}_{\Ellinfty(\R)}+1)\norm{\frac{\cosh(\w_{0}\cdot)}{\cosh(2\w\cdot)}}_{p}.
\end{align}
By equation (5.9) in \cite{Amann97},
\begin{align*}
\Be^{1/p-1/q}_{p,1}(\R;X)=\left(\Ellp(\R;X),\Wonep(\R;X)\right)_{\frac{1}{p}-\frac{1}{q},1}
\end{align*}
with equivalent norms. Moreover, it follows from a direct sum argument that the constant in the norm equivalence does not depend on $X$. Hence \eqref{interpolation inequality}, \eqref{norm iota on Lp} and \eqref{norm iota on Wnp} imply that $\iota:\DAPQ\to\Be^{1/p-1/q}_{p,1}(\R;X)$ is bounded with
\begin{align}\label{norm iota}
\norm{\iota}_{\La(\DAPQ,\Be^{1/p-1/q}_{p,1}(\R;X))}\leq C_{1}M,
\end{align}
for some constant $C_{1}\geq 0$ independent of $A$ and $X$.

It is shown in \cite[Theorem 3.2]{Haase09} that $\ph\ast\psi(s)=\frac{1}{\cosh(\w s)}$ for all $s\in\R$. Let $U_{\mu}\in\La(X)$ be as in \eqref{Hille-Phillips calculus}. Then the abstract transference principle from \cite[Section 2]{Haase11} yields the commutative diagram 
\begin{align*}
\begin{CD}
\Be^{1/p-1/q}_{p,1}(\R;X)	@>L_{\mu_{\w}}>>	\Ellq(\R;X)\\
@A\iota AA										@VVPV\\
\DAPQBIG			 		@>U_{\mu}>>			X
\end{CD}
\end{align*}
of bounded maps. Finally, estimate the norms of $P$ and $\iota$ using \eqref{norm P} and \eqref{norm iota} and apply Corollary \ref{Fourier multiplier measure} to $L_{\mu_{\w}}$.
\end{proof}

\begin{remark}\label{rem:better transference with UMD}
If $X$ is a UMD space then one can replace the space $\D_{A}(\tfrac{1}{p}-\tfrac{1}{q},1)$ by the larger $\D_{A}(\tfrac{1}{p}-\tfrac{1}{q},p)$. This follows in the exact same way as Proposition \ref{transference principle groups interpolation result}, except that one uses \eqref{eq:UMD} instead of Corollary \ref{Fourier multiplier measure} at the very end of the proof.
\end{remark}

For Banach lattices we establish another transference principle. Recall that each Banach space $X$ with type $p=2$ and cotype $q=2$ is isomorphic to an $L^{2}$-space, by \cite{Kwapien72}. Hence the following proposition deals with the case $p=q=2$ in Proposition \ref{transference principle groups interpolation result}.

\begin{proposition}\label{transference principle groups fractional domain result}
Let $\w>\w_{0}\geq 0$, $p\in[1,2]$ and $q\in[2,\infty)$. Let $X$ be isomorphic to a complemented subspace of a $p$-convex and $q$-concave Banach lattice. Then there exists a constant $C\geq 0$ such that the following holds. Let $-\ui A$ generate a $C_{0}$-group $(U(s))_{s\in\R}\subseteq\La(X)$ such that $\norm{U(s)}_{\La(X)}\leq M\cosh(\w_{0}s)$ for all $s\in\R$ and some $M\geq 1$. Then
\begin{align*}
\norm{\int_{\R}U(s)x\,\mu(\ud s)}_{X}\leq CM^{2}\sup_{s\in\R}\,\Big(\abs{s}^{\frac{1}{p}-\frac{1}{q}}\abs{\F\mu_{\w}(s)}\Big)\|x\|_{X}
\end{align*}
for all $\mu\in\M_{\w}(\R)$ and all $x\in X$.
\end{proposition}
\begin{proof}
For $\mu\in\M_{\w}(\R)$ let $U_{\mu}\in\La(X)$ be as in \eqref{Hille-Phillips calculus}. Let $\iota$ and $P$ be as in \eqref{define iota} and \eqref{define P}, with $\psi$ and $\ph$ as in \eqref{define psi and phi}. As in the proof of Proposition \ref{transference principle groups interpolation result} one can factorize $U_{\mu}$ as $U_{\mu}=P\circ L_{\mu_{\w}}\circ \iota$. Hence \eqref{norm iota on Lp}, \eqref{norm P} and Proposition \ref{Fourier multiplier Bessel space} conclude the proof. 
\end{proof}

\section{Results for $C_{0}$-groups}\label{results for groups}

In this section we obtain functional calculus results for generators of $C_{0}$-groups.

\subsection{The main result for $C_{0}$-groups}\label{main result for groups}

We now prove our main functional calculus result for $C_{0}$-groups, already stated in the Introduction as Theorem \ref{main result introduction}. 

\begin{theorem}\label{main theorem}
Let $-\ui A$ generate a $C_{0}$-group $(U(s))_{s\in\R}\subseteq\La(X)$ on a Banach space $X$ with type $p\in[1,2]$ and cotype $q\in[2,\infty)$. Let $\w>\theta(U)$. Then there exists a constant $C\geq 0$ such that $\DAPQBIG\subseteq\D(f(A))$ and
\begin{align}\label{calculus inequality}
\norm{f(A)}_{\La(\DAPQ,X)}\leq C\norm{f}_{\HT^{\infty}\!(\St_{\w})}
\end{align}
for all $f\in\HT^{\infty}\!(\St_{\w})$.
\end{theorem}
\begin{proof}
First consider $f\in\E(\St_{\w})$ and let $\alpha\in(\theta(U),\w)$. By Lemma \ref{decay implies fourier transform} there exists a $\mu\in\M_{\alpha}(\R)$ with $f=\F\mu$ and $f(A)=U_{\mu}$. By Proposition \ref{transference principle groups interpolation result},
\begin{align*}
\norm{f(A)x}_{X}\leq C\norm{\F\mu_{\alpha}}_{\Ellinfty(\R)}\norm{x}_{\frac{1}{p}-\frac{1}{q},1}
\end{align*}
for all $x\in \D_{A}(1/p-1/q,1)$ and some constant $C\geq 0$ independent of $f$ and $x$. This implies \eqref{calculus inequality} since $\norm{\F\mu_{\alpha}}_{\Ellinfty(\R)}\leq \norm{f}_{\HT^{\infty}\!(\St_{\w})}$, as follows from \eqref{Fourier transform perturbed measure}.

For general $f\in\HT^{\infty}\!(\St_{\w})$, define $\tau_{k}(z):=-k^{2}(\ui k-z)^{-2}$ for $k\in\N$ with $k>\w$ and $z\in \St_{\w}$. Then $f\tau_{k}\in \E(\St_{\w})$ for all $k$,
\begin{align*}
\sup_{k} \norm{f\tau_{k}}_{\HT^{\infty}\!(\St_{\w})}\leq \norm{f}_{\HT^{\infty}\!(\St_{\w})}\sup_{k} \norm{\tau_{k}}_{\HT^{\infty}\!(\St_{\w})}<\infty,
\end{align*}
and $(f\tau_{k})(z)\rightarrow f(z)$ as $k\rightarrow \infty$, for all $z\in\St_{\w}$. By what we have already shown,
\begin{align*}
\norm{f\tau_{k}(A)x}_{X}\leq C\norm{f\tau_{k}}_{\HT^{\infty}\!(\St_{\w})}\norm{x}_{\frac{1}{p}-\frac{1}{q},1}\leq C'\norm{f}_{\HT^{\infty}\!(\St_{\w})}\norm{x}_{\frac{1}{p}-\frac{1}{q},1}
\end{align*}
for $C':=C\sup_{k}\norm{\tau_{k}}_{\HT^{\infty}\!(\St_{\w})}$. Since $\D(A^{2})\subseteq\DAPQBIG$ is dense and since $\limsup_{k}\norm{\tau_{k}}_{\HT^{\infty}\!(\St_{\w})}=1$, Lemma \ref{convergence lemma} yields $f(A)\in \La(\DAPQBIG,X)$ with
\begin{align*}
\norm{f(A)}_{\La(\DAPQ,X)}\leq C\norm{f}_{\HT^{\infty}\!(\St_{\w})},
\end{align*}
which concludes the proof.
\end{proof}

\begin{remark}\label{rem:better calculus with UMD}
If in Theorem \ref{main result for groups} one assumes in addition that $X$ is a UMD space, then $\D_{A}(\tfrac{1}{p}-\tfrac{1}{q},1)$ may be replaced by $\D_{A}(\tfrac{1}{p}-\tfrac{1}{q},p)$. This follows by using Remark \ref{rem:better transference with UMD} instead of Proposition \ref{transference principle groups interpolation result} in the proof.
\end{remark}

Under additional assumptions one can obtain stronger results. Note that \eqref{second calculus inequality} improves \eqref{calculus inequality}, since $\D_{A}(\theta,1)\subseteq\D((\lambda+\ui A)^{\theta})$ for each group generator $A$, $\lambda\in\R$ sufficiently large and $\theta\in[0,1]$, by \cite[Proposition 4.1.7]{Lunardi09}.

\begin{theorem}\label{second main theorem}
Let $-\ui A$ generate a $C_{0}$-group $(U(s))_{s\in\R}\subseteq\La(X)$ on a Banach space $X$. Suppose that $X$ is isomorphic to a complemented subspace of a $p$-convex and $q$-concave Banach lattice, for $p\in[1,2]$ and $q\in[2,\infty)$. Let $\lambda>\w>\theta(U)$. Then there exists a constant $C\geq 0$ such that $\D((\lambda+\ui A)^{\frac{1}{p}-\frac{1}{q}})\subseteq\D(f(A))$ and
\begin{align}\label{second calculus inequality}
\norm{f(A)}_{\La(\D((\lambda+\ui A)^{\frac{1}{p}-\frac{1}{q}}),X)}\leq C\norm{f}_{\HT^{\infty}\!(\St_{\w})}
\end{align}
for all $f\in\HT^{\infty}\!(\St_{\w})$.
\end{theorem}
\begin{proof}
Write $\theta:=\frac{1}{p}-\frac{1}{q}$. For $f\in\E(\St_{\w})$ and $\alpha\in(\theta(U),\w)$ there exists a $\mu\in\M_{\alpha}(\R)$ with $f=\F\mu_{\alpha}$ and $f(A)=U_{\mu}$, by Lemma \ref{decay implies fourier transform}. Let $x\in\D((\lambda+\ui A)^{\theta})$. By \cite[Corollary 3.3.6]{Haase06a},
\begin{align*}
f(A)x=f(A)(\lambda+\ui A)^{-\theta}(\lambda+\ui A)^{\theta}x=U_{\mu}U_{\nu}(\lambda+\ui A)^{\theta}x=U_{\mu\ast\nu}(\lambda+\ui A)^{\theta}x,
\end{align*}
where $\nu\in\M_{\w}(\R)$ is given by $\nu(\ud s):=\frac{1}{\Gamma(\theta)}\mathbf{1}_{[0,\infty)}(s)s^{\theta-1}\ue^{-\lambda s}\ud s$. Since $\F(\mu\ast\nu)(s)=f(s)(\lambda+\ui s)^{-\theta}$ for $s\in\R$, \eqref{Fourier transform perturbed measure} yields
\begin{align*}
\sup_{s\in\R}\,\abs{s}^{\theta}\abs{\F(\mu\ast\nu)_{\alpha}(s)}\leq C_{1}\norm{f}_{\HT^{\infty}\!(\St_{\w})}
\end{align*}
for some constant $C_{1}\geq 0$. Now Proposition \ref{transference principle groups fractional domain result} yields a constant $C_{2}\geq 0$ such that
\begin{align*}
\norm{f(A)x}_{X}&\leq C_{2}\sup_{s\in\R}\,\abs{s}^{\theta}\abs{\F(\mu\ast\nu)_{\alpha}(s)}\|(\lambda+\ui A)^{\theta}x\|_{X}\\
&\leq C_{1}C_{2}\norm{f}_{\HT^{\infty}\!(\St_{\w})}\|(\lambda+\ui A)^{\theta}x\|_{X},
\end{align*}
which proves \eqref{second calculus inequality} for $f\in\E(\St_{\w})$. Proceed as in the proof of Theorem \ref{main theorem} to obtain \eqref{second calculus inequality} for general $f\in\HT^{\infty}\!(\St_{\w})$.
\end{proof}

\begin{remark}\label{constant}
It follows from the proofs of Theorems \ref{main theorem} and \ref{second main theorem} that the constants in \eqref{calculus inequality} and \eqref{second calculus inequality} depend on $A$ only through the norm $\norm{U(s)}_{\La(X)}$ of $U(s)$ for all $s\in\R$. In \eqref{calculus inequality} the constant $C$ depends on the underlying space $X$ only through the Gaussian type $p$ constant and the Gaussian cotype $q$ constant of $X$. In particular, the dependence of $C$ on these constants is as in Proposition \ref{Fourier multipliers under type and cotype}. In \eqref{second calculus inequality} the constant depends on the underlying space only through the constant in \eqref{Fourier multiplier inequality Bessel}. Note from the proof of Proposition \ref{Fourier multiplier Bessel space} that, if \eqref{Fourier multiplier inequality Bessel} holds with constant $C\geq0$ on $Y$ and if $X$ is complemented in $Y$ by a projection $P$, then \eqref{Fourier multiplier inequality Bessel} holds on $X$ with constant $C\norm{P}_{\La(Y)}$. This fact will be used in Theorem \ref{second R-bounded calculus}.
\end{remark}

\subsection{More results for $C_{0}$-groups}\label{more results for groups}

Here we derive some additional results for group generators from Theorems \ref{main theorem} and \ref{second main theorem}.

\begin{proposition}\label{resolvent operators}
Let $-\ui A$ generate a $C_{0}$-group $(U(s))_{s\in\R}\subseteq\La(X)$ on a Banach space $X$ with type $p\in[1,2]$ and cotype $q\in[2,\infty)$. Let $\w>\theta(U)$ and $\alpha,\lambda\in\C$ with $\Real(\alpha)>\frac{1}{p}-\frac{1}{q}$ and $\Real(\lambda)>\w$. Then there exists a constant $C\geq 0$ such that $\D((\lambda+\ui A)^{\alpha})\subseteq \D(f(A))$ and
\begin{align*}
\norm{f(A)(\lambda+\ui A)^{-\alpha}}_{\La(X)}\leq C\norm{f}_{\HT^{\infty}\!(\St_{\w})}
\end{align*}
for all $f\in\HT^{\infty}\!(\St_{\w})$.
\end{proposition}
\begin{proof}
The case $p=q=2$ follows from Theorem \ref{second main theorem}. If $\frac{1}{p}-\frac{1}{q}\in(0,1)$ then, by \cite[Propositions 1.1.4 and 4.1.7]{Lunardi09}, $\D((\lambda+\ui A)^{\alpha})\subseteq \DAPQ$ continuously. Hence the proof is concluded by appealing to Theorem \ref{main theorem}.
\end{proof}

One may equivalently formulate Proposition \ref{resolvent operators} as a statement about boundedness on $X$ of the calculus for functions with sufficient decay:

\begin{proposition}\label{functions with decay}
Let $-\ui A$ generate a $C_{0}$-group $(U(s))_{s\in\R}\subseteq\La(X)$ on a Banach space $X$ with type $p\in[1,2]$ and cotype $q\in[2,\infty)$. Let $\w>\theta(U)$ and $\alpha,\lambda\in\C$ with $\Real(\alpha)>\frac{1}{p}-\frac{1}{q}$ and $\Real(\lambda)>\w$. Then there exists a constant $C\geq 0$ such that the following holds. Let $f\in\HT^{\infty}\!(\St_{\w})$ be such that $f(z)\in O(\abs{z}^{-\alpha})$ as $\abs{\Real(z)}\to\infty$. Then $f(A)\in\La(X)$ with
\begin{align}\label{decay estimate}
\norm{f(A)}_{\La(X)}\leq C\!\sup_{z\in\St_{\w}}\abs{\lambda+\ui z}^{\alpha}\abs{f(z)}.
\end{align}
\end{proposition}
\begin{proof}
Apply Proposition \ref{resolvent operators} to $(\lambda+\ui \cdot)^{\alpha}f(\cdot)\in\HT^{\infty}\!(\St_{\w})$.
\end{proof}

In Propositions \ref{resolvent operators} and \ref{functions with decay} one may let $\alpha=\frac{1}{p}-\frac{1}{q}$ if $X$ is isomorphic to a complemented subspace of a $p$-convex and $q$-concave Banach lattice, as follows from Theorem \ref{second main theorem}. 

\begin{corollary}\label{cor:decay order 1/2}
Let $\w>0$ and let $f\in\HT^{\infty}(\St_{\w})$ be such that $f(z)\in O(\abs{z}^{-1/2})$ as $\abs{\Real(z)}\to\infty$. Then $f=\F\mu$, where $\mu\in\M_{\w'}(\R)$ for all $\w'\in[0,\w)$. If in addition $f$ is bounded and holomorphic on $\{z\in\C\mid \Imag(z)>-w\}$ then $\supp(\mu)\subseteq [0,\infty)$. 
\end{corollary}
\begin{proof}
For $\w'\in[0,\w)$, let $A:=\ui \frac{\ud }{\ud t}$ with maximal domain on $X:=L^{1}(\R,\ue^{\w'\abs{t}}\ud t)$. Then $-\ui A$ generates the left translation group $(U(s))_{s\in\R}\subseteq\La(X)$, and $\theta(U)=\w'$. By Proposition \ref{functions with decay}, $f(A)\in\La(X)$. Now \cite[Proposition 2.3]{Haase07} implies that $f=\F\mu$ for some $\mu\in\M_{\w'}(\R)$. By uniqueness of the Fourier transform, $\mu$ is independent of the choice of $\w'\in[0,\w)$. The final statement follows from an application of Liouville's theorem.
\end{proof}

\begin{remark}\label{rem:trivial results}
It follows from Corollary \ref{cor:decay order 1/2} that the conclusion of Theorem \ref{second main theorem} holds without any assumptions on the Banach space $X$ if $\tfrac{1}{p}-\tfrac{1}{q}\geq \tfrac{1}{2}$. Indeed, let $-\ui A$ generate a $C_{0}$-group $(U(s))_{s\in\R}\subseteq\La(X)$ on a general Banach space $X$, and let $\lambda>\w>\theta(U)$ and $f\in\HT^{\infty}(\St_{\w})$. By Corollary \ref{cor:decay order 1/2} and Lemma \ref{decay implies fourier transform}, $f(A)(\lambda+\ui A)^{-1/2}=(f(\cdot)(\lambda+\ui \cdot)^{-1/2})(A)\in\La(X)$ and hence $f(A)\in\La(D((\lambda+\ui A)^{1/2}),X)$. 

Also note that Corollary \ref{cor:decay order 1/2} extends \cite[Lemma 2.4]{Haase-Rozendaal13} from $\alpha>1/2$ to $\alpha=1/2$.
\end{remark}

\begin{remark}\label{does not follow from mikhlin result}
Let $-\ui A$ generate a $C_{0}$-group $(U(s))_{s\in\R}\subseteq\La(X)$ on a UMD space $X$. In \cite{Haase09} it is shown that $A$ has a bounded $\HT^{\infty}_{1}\!(\St_{\w})$-calculus for all $\w>\theta(U)$, where $\HT^{\infty}_{1}\!(\St_{\w})$ is defined by \eqref{mikhlin norm}. Since $X$ has type $p\in(1,2]$ and cotype $q\in[2,\infty)$, one can compare our results with those in \cite{Haase09}. We note that our results do not imply those in \cite{Haase09}, nor does \cite{Haase09} imply the results in this article. Indeed, let $\w>\theta(U)$. Then $f(z):=(\lambda+\ui z)^{-\alpha}$ defines an element of $\HT^{\infty}_{1}\!(\St_{\w})$ for all $\alpha>0$ and $\lambda>\w$, but Proposition \ref{functions with decay} does not apply to $f$ if $\alpha\in(0,\frac{1}{p}-\frac{1}{q})$. Also, the function $f\in\HT^{\infty}(\St_{\w})$ given by $f(z):=\ue^{-\ui z}(\lambda+\ui z)^{-\alpha}$ is not an element of $\HT^{\infty}_{1}(\St_{\w})$ if $\alpha\in(\frac{1}{p}-\frac{1}{q},1)$ but decays with order $\alpha>\frac{1}{p}-\frac{1}{q}$ at infinity.

Although in both examples it is clear that $f(A)\in\La(X)$, the difference between estimating $\norm{f(A)}_{\La(X)}$ by \eqref{mikhlin norm} or using \eqref{decay estimate} is relevant for numerical approximation methods, as is shown in Section \ref{rational approximation}.
\end{remark}

We now obtain a version of Theorem \ref{main theorem} for other interpolation spaces. 

\begin{proposition}\label{other interpolation spaces}
Let $-\ui A$ generate a $C_{0}$-group $(U(s))_{s\in\R}\subseteq\La(X)$ on a Banach space $X$ with type $p\in[1,2]$ and cotype $q\in[2,\infty)$. Let $\w>\theta(U)$, $r\in(0,1-\frac{1}{p}+\frac{1}{q})$ and $u\in[1,\infty]$. Then there exists a constant $C\geq 0$ such that
\begin{align*}
\norm{f(A)x}_{r,u}\leq C\norm{f}_{\HT^{\infty}\!(\St_{\w})}\norm{x}_{r+\frac{1}{p}-\frac{1}{q},u}
\end{align*}
for all $f\in\HT^{\infty}\!(\St_{\w})$ and $x\in\D_{A}(r+\frac{1}{p}-\frac{1}{q},u)$.
\end{proposition}
\begin{proof}
Let $f\in\HT^{\infty}\!(\St_{\w})$ and first consider the case where $u=1$. Then the part $-\ui A_{r,1}$ of $-\ui A$ in $\D_{A}(r,1)$ generates the $C_{0}$-group $(U(s)\!\restriction_{\D_{A}(r,1)})_{s\in\R}\subseteq\La(\D_{A}(r,1))$ which satisfies $\theta(U\!\restriction_{\D_{A}(r,1)})\leq \theta(U)$, by \cite[Lemma 2.2]{Haase-Rozendaal16}. Hence Theorem \ref{main theorem} yields 
\begin{align}\label{estimate other interpolation spaces}
\norm{f(A_{r,1})x}_{r,1}\leq C\norm{f}_{\HT^{\infty}\!(\St_{\w})}\norm{x}_{(\D_{A}(r,1),\D(A_{r,1}))_{\frac{1}{p}-\frac{1}{q},1}}
\end{align}
for all $x\in(\D_{A}(r,1),\D(A_{r,1}))_{\frac{1}{p}-\frac{1}{q},1}$. Moreover, it follows from \cite[Proposition 3.1.5]{Lunardi09} that $\D_{A}(r,1)=(X,\D(A^{2}))_{\frac{r}{2},1}$ and $\D(A_{r,1})=(X,\D(A^{2}))_{\frac{r+1}{2},1}$ with equivalent norms. Hence, by the Reiteration Theorem (see \cite[Theorem 1.3.5]{Lunardi09}) and again by \cite[Proposition 3.1.5]{Lunardi09},
\begin{align*}
(\D_{A}(r,1),\D(A_{r,1}))_{\frac{1}{p}-\frac{1}{q},1}&=((X,\D(A^{2}))_{\frac{r}{2},1},(X,\D(A^{2}))_{\frac{r+1}{2},1})_{\frac{1}{p}-\frac{1}{q},1}\\
&=(X,\D(A^{2}))_{\frac{1}{2}(r+\frac{1}{p}-\frac{1}{q}),1}=\D_{A}(r+\tfrac{1}{p}-\tfrac{1}{q},1).
\end{align*}
Combine this with \eqref{estimate other interpolation spaces}, using that $f(A_{r,1})x=f(A)x$ for all $x\in \D(f(A_{r,1}))$ by \cite[Lemma 2.2]{Haase-Rozendaal16}, to conclude the proof in the case where $u=1$.

For general $u\in[1,\infty]$, the Reiteration Theorem yields
\begin{align*}
\D_{A}(r,u)=(\D_{A}(\theta_{1},1),\D_{A}(\theta_{2},1))_{\theta,u}
\end{align*}
and
\begin{align*}
\D_{A}(r+\tfrac{1}{p}-\tfrac{1}{q},u)=(\D_{A}(\theta_{1}+\tfrac{1}{p}-\tfrac{1}{q},1),\D_{A}(\theta_{2}+\tfrac{1}{p}-\tfrac{1}{q},1))_{\theta,u}
\end{align*}
for certain $\theta_{1},\theta_{2}\in(0,1-\tfrac{1}{p}+\tfrac{1}{q})$ and $\theta\in(0,1)$. Now apply \eqref{interpolation inequality} to what we have already shown to conclude the proof.
\end{proof}

\begin{remark}\label{alternative proof}
For $u<\infty$ Proposition \ref{other interpolation spaces} can be proved directly, without using Theorem \ref{main theorem}. To do so, adapt Proposition \ref{transference principle groups interpolation result} to the setting of Proposition \ref{other interpolation spaces}, using instead of Corollary \ref{Fourier multiplier measure} the more general Proposition \ref{Fourier multipliers under type and cotype}, and then proceed as in the proof of Theorem \ref{main theorem}.
\end{remark}

Similarly, Theorem \ref{second main theorem} extends to other fractional domains.

\begin{proposition}\label{second time other interpolation spaces}
Let $-\ui A$ generate a $C_{0}$-group $(U(s))_{s\in\R}\subseteq\La(X)$ on a complemented subspace $X$ of a $p$-convex and $q$-concave Banach lattice, where $p\in[1,2]$ and $q\in[2,\infty)$. Let $\lambda>\w>\theta(U)$. Then there exists a constant $C\geq 0$ such that
\begin{align*}
\norm{(\lambda+\ui A)^{\alpha}f(A)x}_{X}\leq C\norm{f}_{\HT^{\infty}\!(\St_{\w})}\|(\lambda+\ui A)^{\alpha+\frac{1}{p}-\frac{1}{q}}x\|_{X}
\end{align*}
for all $\alpha\in\C$, $f\in\HT^{\infty}\!(\St_{\w})$ and $x\in\D((\lambda+\ui A)^{\alpha+\frac{1}{p}-\frac{1}{q}})$.
\end{proposition}
\begin{proof}
Apply Theorem \ref{second main theorem} to $(\lambda+\ui A)^{\alpha}x\in\D((\lambda+\ui A)^{\frac{1}{p}-\frac{1}{q}})$ for each $x\in\D((\lambda+\ui A)^{\alpha+\frac{1}{p}-\frac{1}{q}})$.
\end{proof}

\section{Results for sectorial operators and cosine functions}\label{sectorial and cosine}

In this section we derive from Theorem \ref{main theorem} some results for sectorial operators and generators of cosine functions. By Remark \ref{rem:better calculus with UMD}, the results in this section can be improved on UMD spaces, and by Remark \ref{rem:trivial results} the statements hold on general Banach spaces for $\tfrac{1}{p}-\tfrac{1}{q}\geq \tfrac{1}{2}$.

\subsection{Sectorial operators}\label{sectorial operators}

For $\ph\in(0,\pi)$ let $\Se_{\ph}:=\left\{z\in\C\left|\, \abs{\arg(z)}<\ph\right.\right\}$. An operator $A$ on a Banach space $X$ is said to be a \emph{sectorial operator} of angle $\ph$ if $\sigma(A)\subseteq \overline{\Se_{\ph}}$ and $\sup\left\{\norm{zR(z,A)}\,\mid\psi\in\C\setminus \Se_{\psi}\right\}<\infty$ for all $\psi\in (\ph,\pi)$. 

For sectorial operators one can construct a functional calculus in a similar manner as for strip-type operators. Define $f(A)\in\La(X)$ via a Cauchy-type integral for 
\begin{align*}
f\in\HT^{\infty}_{0}\!(\Se_{\psi}):=\Big\{g\in\HT^{\infty}(\Se_{\psi})\Big|\exists C,\delta>0: \abs{g(z)}\leq C\frac{\abs{z}^{\delta}}{\abs{1+z}^{2\delta}}\Big\},
\end{align*}
set $\ind(A):=I_{X}$ and $(1+\cdot)^{-1}(A):=(1+A)^{-1}$, and then extend linearly and regularize as in \eqref{regularization}. For details see \cite[Chapter 2]{Haase06a}. If $A$ is an injective sectorial operator of angle $\ph\in(0,\pi)$, then $\log(A)$ is defined via the sectorial calculus for $A$, as is $f(A)$ for all $\psi\in(\ph,\pi)$ and $f\in \HT^{\infty}\!(\Se_{\psi})$. A sectorial operator $A$ of angle $\ph\in(0,\pi)$ has \emph{bounded imaginary powers} if $A$ is injective and if $-\ui\log(A)$ generates a $C_{0}$-group $(U(s))_{s\in\R}\subseteq\La(X)$. Then $U(s)=A^{-\ui s}$ for all $s\in\R$, and  $A$ is sectorial of angle $\theta_{A}:=\theta(U)$ if $\theta(U)\in [0,\pi)$, by \cite[Theorem 2]{Pruss-Sohr90}. We write $A\in \mathrm{BIP}(X)$. 

If $A\in \mathrm{BIP}(X)$ and if there exists an $\w\in[0,\pi)$ such that $\{\ue^{-\w |s|}A^{-\ui s}\mid s\in\R\}\subseteq\La(X)$ is $R$-bounded, then $A$ has a bounded $\HT^{\infty}\!(\Se_{\psi})$ calculus for all $\psi\in(\w,\pi)$. Moreover, if $X$ has property $(\alpha)$ as below, then a bounded $H^{\infty}(S_{\psi})$-calculus for $A$ implies that $\{\ue^{-\w |s|}A^{-\ui s}\mid s\in\R\}\subseteq\La(X)$ is $R$-bounded for $\w>\psi$. See \cite[Theorem 7.5]{Kalton-Weis04}. Hence the following result is only of use if $A$ does not have $R$-bounded imaginary powers. 

\begin{theorem}\label{BIP theorem}
Let $A\in\mathrm{BIP}(X)$ with $\theta_{A}<\pi$, where $X$ is a Banach space with type $p\in[1,2]$ and cotype $q\in[2,\infty)$. Let $\psi\in(\theta_{A},\pi)$. Then there exists a constant $C\geq 0$ such that $\DLOGPQBIG\subseteq\D(f(A))$ and
\begin{align*}
\norm{f(A)}_{\La(\DLOGPQ,X)}\leq C\norm{f}_{\HT^{\infty}\!(\Se_{\psi})}
\end{align*}
for all $f\in\HT^{\infty}\!(\Se_{\psi})$.
\end{theorem}
\begin{proof}
By \cite[Theorem 4.2.4]{Haase06a}, $(g\circ\log)(A)=g(\log(A))$ for all $g\in\HT^{\infty}\!(\St_{\psi})$. Since $g\mapsto g\circ\log$ is an isometric isomorphism $\HT^{\infty}\!(\St_{\psi})\to\HT^{\infty}\!(\Se_{\psi})$, Theorem \ref{main theorem} concludes the proof.
\end{proof}

\begin{remark}\label{rem:result of Dore}
It was shown by Dore in \cite{Dore01} that each sectorial operator $A$ with dense range on a general Banach space $X$ has a bounded $\HT^{\infty}$-calculus on the real interpolation spaces $(X,\D(A)\cap \ran(A))_{\theta,r}$ for all $\theta\in(0,1)$ and $q\in[1,\infty]$. We note that this statement implies neither Proposition \ref{other interpolation spaces} nor Theorem \ref{BIP theorem}. Indeed, after rotation Proposition \ref{other interpolation spaces} deals with functions on strips around the imaginary axis, whereas for unbounded group generators \cite{Dore01} only applies to $f\in\HT^{\infty}(\Se_{\psi})$ for $\psi>\tfrac{\pi}{2}$. Moreover, for all $\theta\in(0,1)$ and $q\in[1,\infty]$ it holds that $(X,\D(A)\cap \ran(A))_{\theta,r}\subseteq \D(\log(A))$. Since in general $D(\log(A))\subsetneq \DLOGPQBIG$, Theorem \ref{BIP theorem} does not follow from \cite{Dore01}. 
\end{remark}

It follows from Proposition \ref{other interpolation spaces} that $f(A):\D_{\log(A)}(r+\tfrac{1}{p}-\tfrac{1}{q},u)\to\D_{\log(A)}(r,u)$ is bounded for each $r\in(0,1-\tfrac{1}{p}+\tfrac{1}{q})$ and each $u\in[1,\infty]$, and
\begin{align*}
\norm{f(A)}_{\La(\D_{\log(A)}(r+\frac{1}{p}-\frac{1}{q},u),\D_{\log(A)}(r,u))}\leq C\norm{f}_{\HT^{\infty}\!(\Se_{\psi})}.
\end{align*}
In the same manner one can deduce from Proposition \ref{functions with decay} that each $f\in\HT^{\infty}\!(\Se_{\psi})$ such that $f(z)\in O(\abs{2\pi-\ui\log(z)}^{-\alpha})$ as $z\to0$ or $z\to\infty$ for some $\alpha>\frac{1}{p}-\frac{1}{q}$ satisfies $f(A)\in\La(X)$. If $X$ is isomorphic to a complemented subspace of a $p$-convex and $q$-concave Banach lattice then one may let $\alpha=\tfrac{1}{p}-\tfrac{1}{q}$.

From Theorem \ref{BIP theorem} one obtains unconditionality of the functional calculus and square function estimates in the same manner as in \cite[Theorem 12.2]{Kunstmann-Weis04}.

\begin{corollary}\label{unconditionality}
Let $A\in\mathrm{BIP}(X)$ with $\theta_{A}<\pi$, where $X$ is a Banach space with type $p\in[1,2]$ and cotype $q\in[2,\infty)$. Let $\psi\in(\theta_{A},\pi)$, $f\in\HT^{\infty}_{0}\!(\Se_{\psi})$, and let $(r_{k})_{k\in\Z}$ be a Rademacher sequence on a probability space $(\Omega,\mathbb{P})$. Then the following assertions hold:
\begin{itemize}
\item $\sup\{\|\sum_{k=-n}^{n}\epsilon_{k}f(2^{k}tA)\|_{\La(\D_{\log(A)}(\frac{1}{p}-\frac{1}{q},1),X)}\mid n\in\N, t>0,\abs{\epsilon_{k}}=1\}<\infty$;
\item there exists a constant $C\geq 0$ such that 
\begin{align*}
\sup_{t>0}\Big\|\sum_{k=-\infty}^{\infty}r_{k}f(2^{k}tA)x\Big\|_{\Ell^{2}(\Omega;X)}\leq C\norm{x}_{\D_{\log(A)}(\frac{1}{p}-\frac{1}{q},1)}
\end{align*}
for all $x\in\D_{\log(A)}(\frac{1}{p}-\frac{1}{q},1)$, and
\begin{align*}
\sup_{t>0}\Big\|\sum_{k=-\infty}^{\infty}r_{k}f(2^{k}tA)^{*}x^{*}\Big\|_{\Ell^{2}(\Omega;\D_{\log(A)}(\frac{1}{p}-\frac{1}{q},1)^{*})}\leq C\norm{x^{*}}_{X^{*}}
\end{align*}
for all $x^{*}\in X^{*}$.
\end{itemize}
\end{corollary}

\subsection{Cosine functions}\label{cosine functions}

For $\w\geq 0$ let $\varPi_{\w}:=\{z^{2}\mid z\in \St_{\w}\}$. An operator $A$ on a Banach space $X$ is of \emph{parabola-type} $\w$ if $\sigma(A)\subseteq \overline{\varPi_{\w}}$ and if for all $\w'>\w$ there exists a $M_{\w'}\geq 0$ such that
\begin{align*}
\norm{R(\lambda,A)}\leq \frac{M_{\w'}}{\sqrt{\abs{\lambda}}\left(\abs{\Imag(\sqrt{\lambda})}-\w'\right)}\qquad(\lambda\in\C\setminus \varPi_{\w'}).
\end{align*}
For operators of parabola-type $\w\geq0$ there is a natural functional calculus, constructed similarly as the strip-type and sectorial functional calculi, and $f(A)$ is defined as an unbounded operator for all $f\in\HT^{\infty}\!(\varPi_{\w'})$, $\w'>\w$. For details see \cite{Haase13}.

A \emph{cosine function} $\mathrm{Cos}:\R\rightarrow \La(X)$ on a Banach space $X$ is a strongly continuous mapping such that $\mathrm{Cos}(0)=I$ and
\begin{align*}
\mathrm{Cos}(t+s)+\mathrm{Cos}(t-s)=2\mathrm{Cos}(t)\mathrm{Cos}(s)\qquad(s,t\in\R).
\end{align*}
Then
\begin{align*}
\theta(\mathrm{Cos}):=\inf\{\w\geq 0\mid\exists M\geq 0:\norm{\mathrm{Cos}(t)}\leq M\ue^{\w|t|}\textrm{ for all }t\in\R\}<\infty.
\end{align*}
The \emph{generator} of a cosine function is the unique operator $-A$ on $X$ that satisfies
\begin{align*}
\lambda R(\lambda^{2},-A)=\int_{0}^{\infty}\ue^{-\lambda t}\mathrm{Cos}(t)\,\ud t \qquad (\lambda>\theta(\mathrm{Cos})).
\end{align*}
Then $A$ is an operator of parabola-type $\theta(\mathrm{Cos})$. 

We now prove a version of Theorem \ref{main theorem} for generators of cosine functions.

\begin{theorem}\label{cosine function result}
Let $-A$ generate a cosine function $(\mathrm{Cos}(s))_{s\in\R}\subseteq\La(X)$ on a Banach space $X$ with type $p\in[1,2]$ and cotype $q\in[2,\infty)$. Let $\w>\theta(\mathrm{Cos})$. Then there exists a constant $C\geq 0$ such that $\D_{A}(\tfrac{1}{2}(\tfrac{1}{p}-\tfrac{1}{q}),1)\subseteq\D(f(A))$ and
\begin{align*}
\norm{f(A)}_{\La(\D_{A}(\frac{1}{2}(\frac{1}{p}-\frac{1}{q}),1),X)}\leq C\norm{f}_{\HT^{\infty}\!(\varPi_{\w})}
\end{align*}
for all $f\in\HT^{\infty}\!(\varPi_{\w})$. 
\end{theorem}
\begin{proof}
The proof follows the same lines as that of \cite[Proposition 5.5]{Haase-Rozendaal16}. It suffices to assume that $\theta:=\frac{1}{p}-\frac{1}{q}\in(0,1)$. By \cite[Theorem 2]{Kisynski72} there is a unique subspace $V\subseteq X$ such that $\D(A)\subseteq V$ and such that $-\ui\mathcal{A}$ generates a $C_{0}$-group $(U(s))_{s\in\R}\subseteq\La(V\times X)$ on $V\times X$, where
\begin{align*}
\mathcal{A}:=\ui\left[\begin{array}{cc}0&\mathrm{I}_{V}\\-A&0\end{array}\right]
\end{align*}
with domain $\D(\mathcal{A}):=\D(A)\times V$. Moreover, by \cite[Theorem 6.2]{Haase07}, $\theta(\mathrm{Cos})=\theta(U)$. Hence Theorem \ref{main theorem} yields a constant $C\geq 0$ such that $g(\mathcal{A})\in\La(\D_{\mathcal{A}}(\theta,1),V\times X)$ with
\begin{align}\label{estimate g}
\norm{g(\mathcal{A})}_{\La(\D_{\mathcal{A}}(\theta,1),V\times X)}\leq C\norm{g}_{\HT^{\infty}\!(\St_{\w})}
\end{align}
for all $g\in\HT^{\infty}\!(\St_{\w})$.

Let $f\in \HT^{\infty}\!(\varPi_{\w})$. Then $[z\mapsto g(z):=f(z^{2})]\in \HT^{\infty}\!(\St_{\w})$ and $\norm{g}_{\HT^{\infty}\!(\St_{\w})}=\norm{f}_{\HT^{\infty}\!(\varPi_{\w})}$. Moreover, it is straightforward to see that 
\begin{align}\label{direct sum g}
f(A_{V})\oplus f(A)=g(\mathcal{A}).
\end{align}
Now,
\begin{align*}
\mathcal{A}^{2}:=\left[\begin{array}{cc}A_{V}&0\\0&A\end{array}\right]
\end{align*}
with $\D(\mathcal{A}^{2})=\D(A_{V})\times \D(A)$. By \cite[Proposition 3.1.4]{Lunardi09}, 
\begin{align*}
\D(A)\times V\in K_{1/2}\!\left(V\times X,\D(A_{V})\times \D(A)\right)\cap J_{1/2}\!\left(V\times X,\D(A_{V})\times \D(A)\right),
\end{align*}
where $K_{1/2}$ and $J_{1/2}$ are as in \cite[Definition 1.3.1]{Lunardi09}. Hence
\begin{align*}
V\in K_{1/2}\!\left(X,\D(A)\right)\cap J_{1/2}\!\left(X,\D(A)\right).
\end{align*}
Now \cite[Theorem 1.3.5]{Lunardi09} yields 
\begin{align*}
\D_{\mathcal{A}}(\theta,1)&=(V\times X,\D(A)\times V)_{\theta,1}=(V,\D(A))_{\theta,1}\times (X,V)_{\theta,1}\\
&=\D_{A}\left(\tfrac{1+\theta}{2},1\right)\times \D_{A}\left(\tfrac{\theta}{2},1\right).
\end{align*}
Combining this with \eqref{estimate g} and \eqref{direct sum g} yields $f(A)\in\La\left(\D_{A}\left(\tfrac{\theta}{2},1\right),X\right)$ with
\begin{align*}
\norm{f(A)}_{\La\left(\D_{A}\left(\frac{\theta}{2},1\right),X\right)}\leq \norm{g(\mathcal{A})}_{\La(\D_{\mathcal{A}}(\theta,1),V\times X)}\leq C\norm{g}_{\HT^{\infty}\!(\St_{\w})}=C\norm{f}_{\HT^{\infty}\!(\varPi_{\w})},
\end{align*}
as required.
\end{proof}

From Proposition \ref{other interpolation spaces} one deduces in a similar manner that, under the assumptions of Proposition \ref{cosine function result} and for all $r\in(0,1-\tfrac{1}{2}(\tfrac{1}{p}-\tfrac{1}{q}))$ and $u\in[1,\infty]$, there exists a constant $C\geq 0$ such that 
\begin{align*}
\norm{f(A)x}_{r,u}\leq C\norm{f}_{\HT^{\infty}\!(\varPi_{\w})}\norm{x}_{r+\frac{1}{2}(\frac{1}{p}-\frac{1}{q}),u}
\end{align*}
for all $f\in\HT^{\infty}\!(\varPi_{\w})$ and $x\in\D_{A}(r+\tfrac{1}{2}(\tfrac{1}{p}-\tfrac{1}{q}),u)$. We leave the formulation of the obvious analogue of Proposition \ref{functions with decay} for cosine functions to the reader.

\begin{remark}\label{cosine functions on lattices}
Let $-A$ generate a cosine function $(\mathrm{Cos}(s))_{s\in\R}\subseteq\La(X)$, where $X$ is isomorphic to a complemented subspace of a $p$-convex and $q$-concave UMD Banach lattice, for $p\in(1,2]$ and $q\in[2,\infty)$. Let $\lambda>\w> \theta(\mathrm{Cos})$. Then there exists a constant $C\geq 0$ such that
\begin{align*}
\norm{f(A)x}_{X}\leq C\norm{f}_{\HT^{\infty}\!(\varPi_{\w})}\|(\lambda^{2}+A)^{\frac{1}{2}(\frac{1}{p}-\frac{1}{q})}x\|_{X}
\end{align*}
for all $f\in\HT^{\infty}\!(\varPi_{\w})$ and $x\in \D((\lambda^{2}+A)^{\frac{1}{2}(\frac{1}{p}-\frac{1}{q})})$. This follows from Theorem \ref{second main theorem} as in the proof of Proposition \ref{cosine function result}, using the complex interpolation method and \cite[Theorem 6.6.9]{Haase06a} and \cite[Theorem 5.5]{Haase09}.

It should be noted that generators of cosine functions on UMD spaces have a bounded sectorial $\HT^{\infty}$-calculus, by \cite[Theorem 5.5]{Haase09}. Hence on UMD spaces Theorem \ref{cosine function result} is only of use when it does not suffice to obtain an estimate for $\norm{f(A)x}_{X}$ with respect to the supremum norm of $f$ on a sector. The latter is e.g.~the case if $f(z)=g(z^{2})$ for a $g\in\HT^{\infty}(\St_{\w})$ which is unbounded on any double sector $\Se_{\psi}\cup -\Se_{\psi}$ for $\psi\in(0,\tfrac{\pi}{2})$, such as $g(z)=\ue^{-\ui z}$.
\end{remark}

\vanish{By Remark \ref{cosine functions on lattices} the case $\alpha=\frac{1}{2}(\frac{1}{p}-\frac{1}{q})$ is attained on complemented subspaces of $p$-convex and $q$-concave UMD Banach lattices.

\begin{corollary}\label{functions with decay cosine}
Let $-A$ generate a cosine function $(\mathrm{Cos}(s))_{s\in\R}\subseteq\La(X)$ on a Banach space $X$ with type $p\in[1,2]$ and cotype $q\in[2,\infty)$. Let $\alpha>\tfrac{1}{2}(\tfrac{1}{p}-\tfrac{1}{q})$ and $\lambda>\w>\theta(U)$. Then there exists a constant $C\geq 0$ such that $f(A)\in\La(X)$ with 
\begin{align}
\norm{f(A)}_{\La(X)}\leq C\!\sup_{z\in\varPi_{\w}}\abs{\lambda^{2}+z}^{\alpha}\abs{f(z)}
\end{align}
for each $f\in\HT^{\infty}\!(\varPi_{\w})$ such that $f(z)\in O(\abs{z}^{-\alpha})$ as $\abs{z}\to\infty$.
\end{corollary}
}

\section{Operator-valued functional calculus}\label{operator-valued calculus}

In this section we extend our main functional calculus theorems to $R$-bounded operator-valued calculi. Since many of the ideas and proofs are similar to those in the sections before, we leave some details to the reader.

Let $A$ be a strip-type operator of height $\w_{0}\geq 0$ on a Banach space $X$. Let $\A\subseteq\La(X)$ be the algebra of bounded operators that commute with $A$, and let $\w>\w_{0}$. For a bounded holomorphic $f:\St_{\w}\to\A$ such that $\norm{f(z)}_{\La(X)}\leq C\abs{z}^{-\alpha}$ for all $z\in\St_{\w}$ and certain $C\geq 0$ and $\alpha>1$, define $f(A)$ as in \eqref{Cauchy integral}. Regularization as in \eqref{regularization} extends this functional calculus to all bounded holomorphic $f:\St_{\w}\to\A$.

For $\w\geq 0$ let $\M_{\w}(\R;\A)$ consist of the $\A$-valued Borel measures $\mu$ on $\R$ such that $\ue^{\w \abs{s}}\mu(\ud s)$ has bounded variation. If $-\ui A$ generates a $C_{0}$-group $(U(s))_{s\in\R}\subseteq\La(X)$ then one can define $U_{\mu}\in\La(X)$ as in \eqref{Hille-Phillips calculus} for all $\w>\theta(U)$ and $\mu\in\M_{\w}(\R;\A)$. Versions of Lemmas \ref{decay implies fourier transform} and \ref{convergence lemma} hold for this operator-valued calculus, with the same proofs.

For $\w>0$ and $\mathcal{T}\subseteq\La(X)$ let $R\HT^{\infty}\!(\St_{\w};\mathcal{T})$ be the collection of bounded holomorphic functions $f:\St_{\w}\to\mathcal{T}$ such that $\{f(z)\mid z\in\St_{\w}\}\subseteq\La(X)$ is $R$-bounded. If $\mathcal{T}$ is an algebra then $R\HT^{\infty}\!(\St_{\w};\mathcal{T})$ is a Banach algebra with the norm
\begin{align*}
\norm{f}_{R\HT^{\infty}\!(\St_{\w};\mathcal{T})}:=R_{X}(\{f(z)\mid z\in\St_{\w}\})\qquad(f\in R\HT^{\infty}\!(\St_{\w};\mathcal{T})).
\end{align*}
The following result generalizes Theorem \ref{main theorem}, since each $f\in\HT^{\infty}\!(\St_{\w})$ defines an element $\tilde{f}\in R\HT^{\infty}\!(\St_{\w};\A)$ by $\widetilde{f}(z)=f(z)I_{X}$ for $z\in\St_{\w}$, and $\|\widetilde{f}\|_{R\HT^{\infty}\!(\St_{\w};\A)}=\|f\|_{\HT^{\infty}\!(\St_{\w})}$.

\begin{proposition}\label{main operator-valued proposition}
Let $-\ui A$ generate a $C_{0}$-group $(U(s))_{s\in\R}\subseteq\La(X)$ on a Banach space $X$ with type $p\in[1,2]$ and cotype $q\in[2,\infty)$, and let $\w>\theta(U)$. Then there exists a constant $C\geq 0$ such that $\DAPQBIG\subseteq\D(f(A))$ and
\begin{align*}
\norm{f(A)}_{\La(\DAPQ,X)}\leq C\norm{f}_{R\HT^{\infty}\!(\St_{\w};\A)}
\end{align*}
for all $f\in R\HT^{\infty}\!(\St_{\w};\A)$.
\end{proposition}
\begin{proof}
First extend Proposition \ref{transference principle groups interpolation result} to all $\mu\in\M_{\w}(\R;\A)$ for which $R_{X}(\{\F\mu_{\w}(s)\mid s\in\R\})<\infty$. To this end, note that the abstract transference principle in \cite[Section 2]{Haase11} extends to $\A$-valued measures, and appeal to Proposition \ref{Fourier multipliers under type and cotype} instead of Corollary \ref{Fourier multiplier measure}. Then proceed as in the proof of Theorem \ref{main theorem}.
\end{proof}

The other results in the previous sections can be extended to statements about operator-valued calculi in a similar manner. In particular, if in Proposition \ref{main operator-valued proposition} $X$ is isomorphic to a complemented subspace of a $p$-convex and $q$-concave Banach lattice, then for each $\lambda>\w$ there exists a constant $C\geq 0$ such that $\D((\lambda+\ui A)^{\frac{1}{p}-\frac{1}{q}})\subseteq\D(f(A))$ and
\begin{align}\label{operator-valued fractional inequality}
\norm{f(A)}_{\La(\D((\lambda+\ui A)^{\frac{1}{p}-\frac{1}{q}}),X)}\leq C\norm{f}_{R\HT^{\infty}\!(\St_{\w};\A)}
\end{align}
for all $f\in R\HT^{\infty}\!(\St_{\w};\A)$. Remark \ref{constant} applies to Proposition \ref{main operator-valued proposition} and \eqref{operator-valued fractional inequality}.

Let $(r_{k})_{k\in\N}$ and $(r_{j}')_{j\in\N}$ be mutually independent Rademacher sequences on a probability space $(\Omega,\mathbb{P})$. We say that a Banach space $X$ has property $(\alpha)$ if there exists a constant $C\geq 0$ such that, for all $m\in\N$, $\{x_{j,k}\}_{j,k=1}^{m}\subseteq X$ and $\{\alpha_{j,k}\}_{j,k=1}^{m}\subseteq \C$ with $\abs{\alpha_{j,k}}\leq 1$ for all $j,k\in\{1,\ldots,m\}$,
\begin{equation}\label{eq:property alpha}
\Big(\mathbb{E}\mathbb{E}\Big\|\sum_{j,k=1}^{m}r_{k}r_{j}'\alpha_{j,k}x_{j,k}\Big\|^{2}\Big)^{1/2}\leq C\Big(\mathbb{E}\mathbb{E}\Big\|\sum_{j,k=1}^{m}r_{k}r_{j}'x_{j,k}\Big\|^{2}\Big)^{1/2}.
\end{equation}
Each Banach lattice with finite cotype has property $(\alpha)$, and if $X$ has property $(\alpha)$ then so do the closed subspaces of $X$ and $\Ell^{p}(\Omega;X)$ for each $\sigma$-finite measure space $(\Omega,\mu)$ and each $p\in[1,\infty)$.

For $\mathcal{T}$ equal to the unit ball of $\C\subseteq\La(X)$, the following theorem implies that the $\HT^{\infty}\!(\St_{\w})$-calculus for $A$ is $R$-bounded from $\D_{A}(\frac{1}{p}-\frac{1}{q},1)$ to $X$. That is, $\{f(A)\mid \norm{f}_{\HT^{\infty}\!(\St_{\w})}\leq 1\}\subseteq\La(\D_{A}(\frac{1}{p}-\frac{1}{q},1),X)$ is $R$-bounded.

\begin{theorem}\label{R-bounded calculus}
Let $-\ui A$ generate a $C_{0}$-group $(U(s))_{s\in\R}\subseteq\La(X)$ on a Banach space $X$ with property $(\alpha)$, type $p\in[1,2]$ and cotype $q\in[2,\infty)$. Let $\w>\theta(U)$. Then there exists a constant $C\geq 0$ such that
\begin{align*}
R_{\D_{A}(\frac{1}{p}-\frac{1}{q},1),X}(\{f(A)\mid f\in R\HT^{\infty}\!(\St_{\w};\A\cap \mathcal{T})\})\leq C R_{X}(\mathcal{T})
\end{align*}
for each $R$-bounded $\mathcal{T}\subseteq\La(X)$.
\end{theorem}
\begin{proof}
The proof is similar to that of \cite[Theorem 12.8]{Kunstmann-Weis04}. Write $\theta:=\frac{1}{p}-\frac{1}{q}$ and let $(r_{k})_{k\in\N}$ be a Rademacher sequence on $[0,1]$. Fix $n\in\N$ and let $\Rad^{n}(X):=\{\sum_{k=1}^{n}r_{k}x_{k}\mid (x_{k})_{k=1}^{n}\subseteq X\} \subseteq\Ell^{2}([0,1];X)$. Let $\widetilde{A}$ be the operator on $\Rad^{n}(X)$ with domain
\begin{align*}
\D(\widetilde{A})=\Big\{\sum_{k=1}^{n}r_{k}x_{k}\in\Rad^{n}(X)\Big| (x_{n})_{k=1}^{n}\subseteq\D(A)\Big\}
\end{align*}
such that $\widetilde{A}(\sum_{k=1}^{n}r_{k}x_{k}):=\sum_{k=1}^{n}r_{k}Ax_{k}$ for $\sum_{k=1}^{n}r_{k}x_{k}\in\D(\widetilde{A})$. Then $-\ui \widetilde{A}$ generates the $C_{0}$-group $(\widetilde{U}(s))_{s\in\R}\subseteq\La(\Rad^{n}(X))$ given by
\begin{align*}
\widetilde{U}(s)\Big(\sum_{k=1}^{n}r_{k}x_{k}\Big)=\sum_{k=1}^{n}r_{k}U(s)x_{k}
\end{align*}
for $s\in\R$ and $\sum_{k=1}^{n}r_{k}x_{k}\in\Rad^{n}(X)$. Note that $\|\widetilde{U}(s)\|_{\La(\Rad^{n}\!(X))}=\|U(s)\|_{\La(X)}$ for all $s\in\R$. Moreover, $\Rad^{n}(X)\subseteq L^{2}([0,1];X)$ has type $p$ and cotype $q$ with $\tau_{p,\Rad^{n}\!(X)}\leq C_{p}\tau_{p,X}$ and $c_{q,\Rad^{n}\!(X)}\leq C_{q}c_{q,X}$ for constants $C_{p},C_{q}\geq0$ depending only on $p$ and $q$ that come from the Kahane-Khintchine inequalities. By Proposition \ref{main operator-valued proposition} and Remark \ref{constant} there exists a constant $C_{1}\geq 0$ independent of $n$ such that
\begin{align}\label{bound for Atilde}
\|f(\widetilde{A})\|_{\La(\D_{\widetilde{A}}(\theta,1),\Rad^{n}\!(X))}\leq C_{1}\norm{f}_{R\HT^{\infty}\!(\St_{\w};\widetilde{\mathcal{A}})}
\end{align}
for all $f\in R\HT^{\infty}\!(\St_{\w};\widetilde{\mathcal{A}})$, where $\widetilde{\mathcal{A}}\subseteq\La(\Rad^{n}(X))$ is the algebra of operators commuting with $\widetilde{A}$.

Let $\mathcal{T}\subseteq\La(X)$ be $R$-bounded and let $f_{1},\ldots, f_{n}\in R\HT^{\infty}\!(\St_{\w};\mathcal{A}\cap\mathcal{T})$. Define
\begin{align*}
f(z)\Big(\sum_{k=1}^{n}r_{k}x_{k}\Big)=\sum_{k=1}^{n}r_{k}f_{k}(z)x_{k}
\end{align*}
for $z\in \St_{\w}$ and $\sum_{k=1}^{n}r_{k}x_{k}\in\Rad^{n}(X)$. We will now show that the range of $f$ is $R$-bounded in $\La(\Rad^{n}(X))$, from which it will follow in particular that $f:\St_{\w}\to\widetilde{\mathcal{A}}$ is well-defined. 

Let $(r_{j}')_{j\in\N}$ be a Rademacher sequence on $[0,1]$, independent of $(r_{k})_{k\in\N}$. Let $m\in\N$, $(z_{j})_{j=1}^{m}\subseteq\St_{\w}$ and $(y_{j})_{j=1}^{m}\subseteq\Rad^{n}(X)$. Write $y_{j}=\sum_{k=1}^{n}r_{k}x_{jk}$ for $j\in\{1,\ldots, m\}$ and $(x_{jk})_{k=1}^{n}\subseteq X$. Then \cite[Lemma 4.11 and Remark 4.10]{Kunstmann-Weis04} yield a constant $C_{2}\geq 0$ depending only on $X$ such that
\begin{align*}
\Big\|\sum_{j=1}^{m}r_{j}'f(z_{j})y_{j}\Big\|_{\Ell^{2}([0,1];\Rad^{n}\!(X))}^{2}&=\int_{0}^{1}\int_{0}^{1}\Big\|\sum_{j=1}^{m}\sum_{k=1}^{n}r_{j}'(v)r_{k}(u)f(z_{j})x_{jk}\Big\|_{X}^{2}\ud u \ud v\\
&\leq C_{2}^{2}R_{X}(\mathcal{T})^{2}\!\int_{0}^{1}\int_{0}^{1}\Big\|\sum_{j=1}^{m}\sum_{k=1}^{n}r_{j}'(v)r_{k}(u)x_{jk}\Big\|_{X}^{2}\ud u\ud v\\
&=C_{2}^{2}R_{X}(\mathcal{T})^{2}\Big\|\sum_{j=1}^{m}r_{j}'y_{j}\Big\|_{\Ell^{2}([0,1];\Rad^{n}\!(X))}^{2}.
\end{align*}
Hence $f\in R\HT^{\infty}\!(\St_{\w};\widetilde{\mathcal{A}})$ with $\norm{f}_{R\HT^{\infty}\!(\St_{\w};\widetilde{\mathcal{A}})}\leq C_{2}R_{X}(\mathcal{T})$. Combining this with \eqref{bound for Atilde} yields
\begin{align}\label{complete bound for Atilde}
\|f(\widetilde{A})\|_{\La(\D_{\widetilde{A}}(\theta,1),\Rad^{m}(X))}\leq C_{1}C_{2}R_{X}(\mathcal{T}).
\end{align}
Note that $\big\|\sum_{k=1}^{n}r_{k}x_{k}\|_{\Rad^{n}\!(\D(A))}\leq \big\|\sum_{k=1}^{n}r_{k}x_{k}\|_{\D(\widetilde{A})}\leq 2\big\|\sum_{k=1}^{n}r_{k}x_{k}\|_{\Rad^{n}\!(\D(A))}$ for all $\sum_{k=1}^{n}r_{k}x_{k}\in \D(\widetilde{A})$, hence $\D_{\widetilde{A}}(\theta,1)=\Rad^{n}(\D_{A}(\theta,1))$ with 
\begin{align}\label{bound for Atilde interpolation}
\Big\|\sum_{k=1}^{n}r_{k}x_{k}\Big\|_{\Rad^{n}\!(\D_{A}(\theta,1))}\leq \Big\|\sum_{k=1}^{n}r_{k}x_{k}\Big\|_{\D_{\widetilde{A}}(\theta,1)}\leq 2\Big\|\sum_{k=1}^{n}r_{k}x_{k}\Big\|_{\Rad^{n}\!(\D_{A}(\theta,1))}
\end{align}
for all $\sum_{k=1}^{n}r_{k}x_{k}\in \D_{\widetilde{A}}(\theta,1)$. Also, it is straightforward to check (by regularization) that $f(\widetilde{A})(\sum_{k=1}^{n}r_{k}x_{k})=\sum_{k=1}^{n}r_{k}f_{k}(A)x_{k}$ for all $\sum_{k=1}^{n}r_{k}x_{k}\in\D(f(\widetilde{A}))$. Hence \eqref{complete bound for Atilde} and \eqref{bound for Atilde interpolation} yield
\begin{align*}
\Big\|\sum_{k=1}^{n}r_{k}f_{k}(A)x_{k}\Big\|_{\Ell^{2}([0,1];X)}&=\Big\|f(\widetilde{A})\Big(\sum_{k=1}^{n}r_{k}x_{k}\Big)\Big\|_{\Rad^{n}\!(X)}\\
&\leq C_{1}C_{2}R_{X}(\mathcal{T})\Big\|\sum_{k=1}^{n}r_{k}x_{k}\Big\|_{\D_{\widetilde{A}}(\theta,1)}\\
&\leq CR_{X}(\mathcal{T})\Big\|\sum_{k=1}^{n}r_{k}x_{k}\Big\|_{\Ell^{2}([0,1];\D_{A}(\theta,1))}
\end{align*}
for all $x_{1},\ldots, x_{n}\in \D_{A}(\theta,1)$, where $C\geq 0$ is independent of $\mathcal{T}\subseteq \La(X)$, $n\in\N$, $f_{1},\ldots f_{n}\in R\HT^{\infty}\!(\St_{\w};\mathcal{A}\cap \mathcal{T})$ and $x_{1},\ldots, x_{n}\in\D_{A}(\theta,1)$. This concludes the proof.
\end{proof}

In the same manner we deduce an $R$-bounded version of Theorem \ref{second main theorem}, under the extra assumption of $p$-convexity for some $p>1$. By \cite[Corollary 1.f.9]{Lindenstrauss-Tzafriri79} the latter is equivalent to the assumption of nontrivial type. Recall that any closed subspace of a Banach lattice with finite cotype has property $(\alpha)$.

\begin{theorem}\label{second R-bounded calculus}
Let $X$ be isomorphic to a complemented subspace of a $p$-convex and $q$-concave Banach lattice, for $p\in(1,2]$ and $q\in[2,\infty)$. Let $-\ui A$ generate a $C_{0}$-group $(U(s))_{s\in\R}\subseteq\La(X)$, and let $\lambda>\w>\theta(U)$. Then there exists a constant $C\geq 0$ such that 
\begin{align*}
R_{\D((\lambda+\ui A)^{\frac{1}{p}-\frac{1}{q}}),X}(\{f(A)\mid f\in R\HT^{\infty}\!(\St_{\w};\A\cap \mathcal{T})\})\leq C R_{X}(\mathcal{T})
\end{align*}
for each $R$-bounded $\mathcal{T}\subseteq\La(X)$.
\end{theorem}
\begin{proof}
We use notation as in the proof of Theorem \ref{R-bounded calculus}. It suffices to show that there exists a constant $C\geq 0$ independent of $n\in\N$ such that 
\begin{align}\label{second bound for Atilde}
\|f(\widetilde{A})\|_{\La(\D((\lambda+\ui \widetilde{A})^{\frac{1}{p}-\frac{1}{q}}),\Rad^{n}\!(X))}\leq C\norm{f}_{R\HT^{\infty}\!(\St_{\w};\widetilde{\mathcal{A}})}
\end{align}
for all $f\in R\HT^{\infty}\!(\St_{\w};\widetilde{\mathcal{A}})$. Indeed, once this has been established, the rest of the proof is identical to that of Theorem \ref{R-bounded calculus}. To obtain \eqref{second bound for Atilde} apply \eqref{operator-valued fractional inequality} to $\widetilde{A}$ on $\Rad^{n}(X)$. To see that the constant $C$ that one gets from this can be chosen to be independent of $n$, it suffices by Remark \ref{constant} to show that $\Rad^{n}(X)$ is complemented in $\Ell^{2}([0,1];X)$ by a projection $P_{n}\in\La(\Ell^{2}([0,1];X))$ with $\norm{P_{n}}_{\La(\Ell^{2}([0,1];X))}\leq C'$ for some $C'\geq 0$ independent of $n$. The latter in turn follows from the fact that $X$ has nontrivial type and from Pisier's characterization in \cite{Pisier82} of the $K$-convex Banach spaces as the spaces with nontrivial type.
\end{proof}

We do not know whether the assumption in Theorem \ref{second R-bounded calculus} that $X$ has nontrivial type is necessary.

From Theorems \ref{R-bounded calculus} and \ref{second R-bounded calculus} one can deduce $R$-bounded versions of Theorems \ref{BIP theorem} and \ref{cosine function result} in the obvious manner. Also, as a corollary of our results, for $C_{0}$-groups we improve Theorem \cite[Theorem 6.1]{Hytonen-Veraar09}.

\begin{corollary}\label{R-bounded group}
Let $-\ui A$ generate a $C_{0}$-group $(U(s))_{s\in\R}\subseteq\La(X)$ on a Banach space $X$ with property $(\alpha)$, type $p\in[1,2]$ and cotype $q\in[2,\infty)$. Let $\w>\theta(U)$. Then $\{\ue^{-\w\abs{s}}U(s)\mid s\in\R\}\subseteq\La(\D_{A}(\frac{1}{p}-\frac{1}{q},1),X)$ is $R$-bounded.

If $X$ is isomorphic to a complemented subspace of a $p$-convex and $q$-concave Banach lattice for $p\in(1,2]$ and $q\in[2,\infty)$, then $\{\ue^{-\w\abs{s}}U(s)\mid s\in\R\}\subseteq\La(\D(\lambda+\ui A)^{\frac{1}{p}-\frac{1}{q}},X)$ is $R$-bounded for each $\lambda>\w$.
\end{corollary}
\begin{proof}
Let $\w'\in(\theta(U),\w)$ and let $\mathcal{T}$ be the unit ball of $\C\subseteq\La(X)$. Now apply Theorems \ref{R-bounded calculus} and \ref{second R-bounded calculus} to $\{\ue^{-(\w\abs{s}+\ui s\cdot)}\mid s\in\R\}\subseteq R\HT^{\infty}(\St_{\w'};\mathcal{A}\cap \mathcal{T})$.
\end{proof}

As the following example shows, Corollary \ref{R-bounded group} and Theorem \ref{second R-bounded calculus} are sharp.

\begin{example}\label{sharpness example}
Let $p\in[1,\infty)$ and let $(U(s))_{s\in\R}\subseteq\La(X)$ be the left translation group on $X:=\Ell^{p}(\R)$ with generator $-\ui A$, where $Af:=\ui f'$ for $f\in\D(A)=\W^{1,p}(\R)$. Then $\D((\ui A)^{\alpha})=\mathrm{H}^{\alpha,p}(\R)$, where $\mathrm{H}^{\alpha,p}(\R)$ is a Bessel-potential space. It is shown in \cite[Example 6.2]{Hytonen-Veraar09} that $\{U(s)\mid s\in[-1,1]\}\subseteq \La(\mathrm{H}^{\alpha,p}(\R),\Ell^{p}(\R))$ is not $R$-bounded for $\alpha\in[0,\frac{1}{p}-\frac{1}{q})$. Hence $\{\ue^{-\w\abs{s}}U(s)\mid s\in\R\}\subseteq \La(\D((\ui A)^{\alpha}),X)$ is not $R$-bounded for $w\in\R$ and $\alpha\in[0,\frac{1}{p}-\frac{1}{q})$.
\end{example}

\section{Rational approximation}\label{rational approximation}

In this section we give an application of the results in previous sections to the theory of rational approximation of $C_{0}$-groups. Note that the results in Section \ref{operator-valued calculus} can be used to replace the uniform bounds in this section by $R$-bounds. 

Recall that a $C_{0}$-semigroup $(T(t))_{t\geq 0}\subseteq\La(X)$ on a Banach space $X$ is \emph{exponentially stable} if there exist $M\geq1$ and $\w>0$ such that $\norm{T(t)}_{\La(X)}\leq M\ue^{-\w t}$ for all $t\geq 0$. We note that, if $-A$ generates an exponentially stable $C_{0}$-semigroup $(T(t))_{t\geq 0}\subseteq\La(X)$ such that $T(t)\in\La(X)$ is invertible for each $t\geq 0$, then $-A$ in fact generates the $C_{0}$-group $(U(s))_{s\in\R}\subseteq\La(X)$, where $U(s):=T(s)$ for $s\geq 0$ and $U(s):=T(-s)^{-1}$ for $s<0$. Then $f(A)$ is defined as an unbounded operator for each $f\in\HT^{\infty}\!(\C_{+})$ by a shifted version of the strip-type calculus from Section \ref{functional calculus}.

\begin{lemma}\label{half-plane estimate}
Let $-A$ generate an exponentially stable $C_{0}$-semigroup $(T(t))_{t\geq 0}$ on a Banach space $X$ with type $p\in[1,2]$ and cotype $q\in[2,\infty)$. Suppose that $T(t)$ is invertible for all $t\geq0$. Then there exists a constant $C\geq 0$ such that
\begin{align}\label{half-plane interpolation estimate}
\norm{f(A)}_{\La(\D_{A}(\frac{1}{p}-\frac{1}{q},1),X)}\leq C\norm{f}_{\HT^{\infty}\!(\C_{+})}
\end{align}
for all $f\in\HT^{\infty}\!(\C_{+})$. For each $\beta>\tfrac{1}{p}-\tfrac{1}{q}$ there exists a constant $C'\geq 0$ such that
\begin{align}\label{half-plane fractional estimate}
\norm{f(A)A^{-\beta}}_{\La(X)}\leq C'\norm{f}_{\HT^{\infty}\!(\C_{+})}
\end{align}
for all $f\in\HT^{\infty}\!(\C_{+})$. 
\end{lemma}
\begin{proof}
Let $f\in\HT^{\infty}\!(\C_{+})$ and apply Theorem \ref{main theorem} and Proposition \ref{resolvent operators} to the strip-type operator $-\ui(A-\w)$ and the function $f(\ui\cdot+\w)\in\HT^{\infty}\!(\St_{\w})$ for $\w>0$ large enough. Then use the composition rule
\begin{align*}
f(\ui\cdot+\w)(-\ui(A-\w))=f(A),
\end{align*}
which is straightforward to prove in the same manner as \cite[Theorem 4.2.4]{Haase06a}.
\end{proof}

\vanish{
\begin{remark}\label{only strip needed}
Note that in Lemma \ref{half-plane estimate} one does not actually need the function $f$ to be defined or bounded on all of $\C_{+}$, merely on a strip $\left\{z\in\C\mid \w_{1}<\Real(z)<\w_{2}\right\}$ for suitable $\w_{1},\w_{2}\in\R$. The estimates in \eqref{half-plane interpolation estimate} and \eqref{half-plane fractional estimate} then hold with respect to the supremum norm of $f$ on this strip.

In the same manner, in Lemma \ref{other approximation methods} one can restrict to functions which are bounded in supremum norm by $1$ on a suitable strip.
\end{remark}
}

Lemma \ref{half-plane estimate} applies to the important question of the power-boundedness of the \emph{Cayley transform} $(1-A)(1+A)^{-1}$ of $A$, which in turn is equivalent to the stability of the Crank-Nicholson approximation scheme associated with $A$.

\begin{corollary}\label{Cayley transform}
Let $-A$ generate an exponentially stable $C_{0}$-semigroup $(T(t))_{t\geq 0}\subseteq\La(X)$ on a Banach space $X$ with type $p\in[1,2]$ and cotype $q\in[2,\infty)$. Suppose that $T(t)$ is invertible for all $t\geq0$. Then
\begin{align*}
\sup_{n\in\N}\norm{(1-A)^{n}(1+A)^{-n}}_{\La(\D_{A}(\frac{1}{p}-\frac{1}{q},1),X)}<\infty.
\end{align*}
\end{corollary}

For $n\in\N$ let $p_{n}$ and $q_{n}$ be the unique polynomials of degree $n$ and $n+1$, respectively, such that $p_{n}(0)=q_{n}(0)=1$ and such that
\begin{align*}
\left|\frac{p_{n}(z)}{q_{n}(z)}-\ue^{z}\right|\leq C\abs{z}^{2n+2}\end{align*}
for all $z$ in a neighborhood of $0\in\C$. Let $r_{n}:=\frac{p_{n}}{q_{n}}$. Then $r_{n}$ is the $n$-th \emph{subdiagonal Pad\'{e} approximation} of the exponential function. 

For $\frac{1}{p}-\frac{1}{q}\in[0,\frac{1}{2})$ the following proposition improves convergence rates obtained in \cite[Theorem 4.1]{Egert-Rozendaal13} for uniformly bounded $C_{0}$-semigroups on general Banach spaces. Note that, on a Banach space $X$ with type $p\in[1,2]$ and cotype $q\in[2,\infty)$, for $\alpha>\tfrac{1}{p}-\tfrac{1}{q}$ we obtain strong convergence of $r_{n}(-tA)$ to $T(t)$ on $\D(A^{\alpha})$ with rate $\cap_{a<\alpha}\,O(n^{-a+\frac{1}{p}-\frac{1}{q}})$, locally uniformly in $t$.

\begin{proposition}\label{convergence rates}
Let $-A$ generate an exponentially stable $C_{0}$-semigroup $(T(t))_{t\geq 0}$ on a Banach space $X$ with type $p\in[1,2]$ and cotype $q\in[2,\infty)$. Suppose that $T(t)$ is invertible for all $t\geq0$. Let $\alpha>\tfrac{1}{p}-\tfrac{1}{q}$ and $a\in(0,\alpha-\tfrac{1}{p}+\tfrac{1}{q})$. Then there exists a constant $C\geq 0$ such that
\begin{align*}
\norm{r_{n}(-tA)x-T(t)x}_{X}\leq Ct^{a}(n+1)^{-a}\norm{A^{\alpha}x}_{X}
\end{align*}
for all $t\in(0,\infty)$, all $n\in\N$ with $n> \frac{\alpha}{2}-1$ and all $x\in\D(A^{\alpha})$. Moreover, $(r_{n}(-tA))_{n\in\N}$ converges strongly on $\D_{A}(\tfrac{1}{p}-\tfrac{1}{q},1)$ to $T(t)$, locally uniformly in $t$.
\end{proposition}
\begin{proof}
Let $t\in(0,\infty)$ and $n\in \N$ with $n\geq \frac{\alpha}{2}-1$. Set 
\begin{align*}
f(z):=\frac{r_{n}(-tz)-\ue^{-tz}}{z^{a}}\qquad (z\in\C_{+}).
\end{align*}
Then Lemma \ref{half-plane estimate} yields a constant $C'\geq 0$ such that
\begin{align*}
\norm{(r_{n}(-tA)-T(t))A^{-\alpha}}_{\La(X)}=\norm{f(A)A^{-\alpha+a}}_{\La(X)}\leq C'\norm{f}_{\HT^{\infty}\!(\C_{+})}.
\end{align*}
By \cite[Lemmas 3.3 and 3.5]{Egert-Rozendaal13},
\begin{align*}
\norm{f}_{\HT^{\infty}\!(\C_{+})}=t^{a}\!\sup_{z\in\C_{+}}\frac{r_{n}(-tz)-\ue^{-tz}}{(tz)^{a}}\leq 2t^{a}(n+1)^{-a}.
\end{align*}
Therefore, with $C:=2C'$,
\begin{align*}
\norm{r_{n}(-tA)x-T(t)x}_{X}&\leq \norm{(r_{n}(-tA)-T(t))A^{-\alpha}}_{\La(X)}\norm{A^{\alpha}x}_{X}\\
&\leq Ct^{a}(n+1)^{-a}\norm{A^{\alpha}x}_{X}
\end{align*}
for all $x\in\D(A^{\alpha})$, which proves the first statement.

Since $\norm{r_{n}}_{\HT^{\infty}\!(\C_{-})}\leq 1$ for all $n\in\N$ by \cite{Ehle73}, Lemma \ref{half-plane estimate} yields that
\begin{align*}
\left\{r_{n}(-tA)-T(t)\mid n\in\N,t\geq 0\right\}\subseteq\La(\D_{A}(\tfrac{1}{p}-\tfrac{1}{q},1),X)
\end{align*}
is uniformly bounded. The proof is now concluded by what we have already shown and by the fact that $\D(A^{2})$ is dense in $\D_{A}(\frac{1}{p}-\frac{1}{q},1)$.
\end{proof}

The same method that was used in Proposition \ref{convergence rates} to yield strong convergence on $\D_{A}(\frac{1}{p}-\frac{1}{q},1)$ also works for other rational approximation methods. Recall that a rational function $r\in\HT^{\infty}\!(\C_{-})$ is said to be \emph{$\mathcal{A}$-stable} if $\norm{r}_{\HT^{\infty}\!(\C_{-})}\leq 1$, and $r$ is a \emph{rational approximation (of the exponential function) of order $k\in\N$} if there exists a constant $C\geq 0$ such that $\abs{r(z)-\ue^{z}}\leq C\abs{z}^{k+1}$ for all $z$ in a complex neighborhood of $0$.

\begin{corollary}\label{other approximation methods}
Let $r$ be an $\mathcal{A}$-stable rational approximation of order $k\in\N$. Let $-A$ generate an exponentially stable $C_{0}$-semigroup $(T(t))_{t\geq 0}$ on a Banach space $X$ with type $p\in[1,2]$ and cotype $q\in[2,\infty)$. Suppose that $T(t)$ is invertible for all $t\geq0$. Then $(r(-\tfrac{t}{n}A)^{n})_{n\in\N}$ converges strongly on $\D_{A}(\frac{1}{p}-\frac{1}{q},1)$ to $T(t)$, locally uniformly in $t\geq 0$. 
\end{corollary}
\begin{proof}
Lemma \ref{half-plane estimate} yields a constant $C\geq 0$ such that
\begin{align*}
\norm{r(-\tfrac{t}{n}A)^{n}-T(t)}_{\La(\D_{A}(\frac{1}{p}-\frac{1}{q},1),X)}&\leq C\norm{r(-\tfrac{t}{n}\cdot)^{n}-\ue^{-t\cdot}}_{\HT^{\infty}\!(\C_{+})}\\
&\leq C(\norm{r^{n}}_{\HT^{\infty}\!(\C_{-})}+1)\leq 2C
\end{align*}
for all $n\in\N$ and $t\geq 0$. Since, by \cite[Theorem 3]{Brenner-Thomee79}, $r(-\tfrac{t}{n}A)^{n}$ converges locally uniformly in $t$ to $T(t)$ on $\D(A^{k+1})$, the uniform boundedness of 
\begin{align*}
\left\{r(-\tfrac{t}{n}A)^{n}-T(t)\mid t\geq 0,n\in\N\right\}\subseteq\La(\D_{A}(\tfrac{1}{p}-\tfrac{1}{q},1),X)
\end{align*}
and the fact that $\D(A^{k+1})$ is dense in $\D_{A}(\tfrac{1}{p}-\tfrac{1}{q},1)$ yield the desired statement.
\end{proof}

\begin{remark}\label{rates on Lp}
If $X$ is isomorphic to a complemented subspace of a $p$-convex and $q$-concave Banach lattice, for $p\in[1,2]$ and $q\in[2,\infty)$, then the case $\beta=\frac{1}{p}-\frac{1}{q}$ is attained in Lemma \ref{half-plane estimate}. Hence, in the setting of Corollary \ref{Cayley transform},
\begin{align*}
\sup_{n\in\N}\|(1-A)^{n}(1+A)^{-n-\frac{1}{p}+\frac{1}{q}\!}\|_{\La(X)}<\infty.
\end{align*}
Moreover, one obtains rate $O(n^{-\alpha+\frac{1}{p}-\frac{1}{q}})$ in Proposition \ref{convergence rates}, and strong convergence on $\D(A^{\frac{1}{p}-\frac{1}{q}})$ in Corollary \ref{other approximation methods}.
\end{remark}

\subsubsection{Funding}

This work was supported by the Netherlands Organisation for Scientific Research (NWO) [grant number 613.000.908 ``Applications of Transference Principles''].

\subsubsection{Acknowledgements}
 
The author thanks Mark Veraar for numerous helpful suggestions, and the referee for carefully checking the manuscript.

\bibliographystyle{plain}
\bibliography{Bibliografie}

\end{document}